\documentclass[11pt,leqno]{amsart}

\RequirePackage{amsthm,amsmath,amsfonts}

\usepackage{color}
\usepackage{amsmath}
\usepackage{amssymb}
\usepackage{amsthm}
\usepackage{amsfonts}
\usepackage{bbm}
\usepackage{dsfont}
\usepackage{verbatim}
\usepackage{mathtools}

\usepackage{graphicx}
\usepackage[left=2.6cm, right=2.6cm, top=2.6cm, bottom= 2.6cm]{geometry}

\numberwithin{equation}{section}

\usepackage[colorlinks=true, pdfstartview=FitV, linkcolor=blue, citecolor=blue, urlcolor=blue,pagebackref=false]{hyperref}

\newcommand{\bN}{\mathbf{N}}

\newcommand{\bP}{\mathbb{P}}
\newcommand{\Q}{\mathbb{Q}}

\newcommand{\E}{\mathbb{E}}
\newcommand{\N}{\mathbb{N}}

\newcommand{\R}{\mathbb{R}}

\newcommand{\eps}{\epsilon}

\newcommand{\cM}{\mathcal{M}}
\newcommand{\cG}{\mathcal{G}}

\newcommand{\cF}{\mathcal{F}}

\newcommand{\cA}{\mathcal{A}}

\newcommand{\cL}{\mathcal{L}}

\newcommand{\cP}{\mathcal{P}}

\newcommand{\cB}{\mathcal{B}}

\def\restrict#1{\raise-.5ex\hbox{\ensuremath|}_{#1}}

\newcommand\restr[2]{{
  \left.\kern-\nulldelimiterspace 
  #1 
  \vphantom{\big|} 
  \right|_{#2} 
  }}

\newtheorem{theorem}{Theorem}[section]
\newtheorem{lemma}[theorem]{Lemma}
\newtheorem{proposition}[theorem]{Proposition}

\usepackage{chngcntr}
\counterwithout{theorem}{section}

\begin{document}

\title{Instantaneous support propagation for $\Lambda$-Fleming-Viot processes}
\author{Thomas Hughes and Xiaowen Zhou}
\maketitle

\begin{abstract}
For a probability-measure-valued neutral Fleming-Viot process $Z_t$ with L\'evy mutation and resampling mechanism associated to a general $\Lambda$-coalescent with multiple collisions, we prove the instantaneous propagation of supports. That is, at any fixed time $t>0$, with probability one the closed support $S(Z_t)$ of the Fleming-Viot process satisfies $S(\nu * Z_t) \subseteq S(Z_t)$, where $\nu$ is the L\'evy measure of the mutation process. To show this result, we apply Donnelly-Kurtz's lookdown particle representation for Fleming-Viot processes.
\end{abstract}

\section{Introduction and main results} \label{s_intro}

\subsection{Introduction}
In this work, we study generalized Fleming-Viot processes, also called $\Lambda$-Fleming-Viot processes, a class of probability measure-valued Markov processes which model the evolving distribution of genetic types in a population subject to random reproduction, mutation and genetic (allele) drift. The models considered are selectively neutral, meaning that no genetic type has a reproductive advantage over another. In the context of measure-valued Markov processes, it is natural to view mutation as a spatial motion on the space of genetic types and to study the measure's support properties vis-\`a-vis the mutation/motion. It is such a perspective that we take here, in particular the case in which type space is $\R^d$ and the mutation operator is the generator of a L\'evy process with jumps.

The classical Fleming-Viot model was introduced by Fleming and Viot in \cite{FV1979}. The model and its variants, both with and without mutation, have been studied in depth. For a survey of some classical results, see \cite{EK1993}. One feature of the model is its connection to Kingman's coalescent, an exchangeable integer partition-valued Markov process arising in mathematical population genetics. The coalescent describes the merging of ancestral lines backwards in time of samples from the Fleming-Viot process. The connection was made rigorous in \cite{DK96} via a construction of the Fleming-Viot process explicitly carrying versions of Kingman's coalescent. The more general $\Lambda$-Fleming-Viot processes dual to more general coalescents can be found in Donnelly and Kurtz \cite{DK99} who identified them as a class of measure-valued dual processes to the $\Lambda$-coalescents, which themselves are a class of exchangeable integer partition-valued Markov processes generalizing Kingman's coalescent by allowing multiple collisions. The $\Lambda$-coalescents were introduced independently by Pitman \cite{P1999} and Sagitov \cite{S1999}. The class of $\Lambda$-Fleming-Viot processes allows more general reproduction mechanisms in which each individual can possibly give birth to a large number of children that is comparable to the size of the whole population. Finally, we note that Fleming-Viot-type processes (both classical and $\Lambda$-) arise as the infinite particle limits of individual-based Markov models for the distribution of genetic types in a population of fixed size, such as the Moran process and its variants such as Cannings population models with non-overlapping generations; see Donnelly and Kurtz \cite{DK96, DK99}. We refer to Bertoin and Le Gall \cite{BLG2003,BLG2005,BG2006} for an alternative representation of the mutationless $\Lambda$-Fleming-Viot process via a flow of bridges.

An even more general class of Fleming-Viot processes that is dual to $\Xi$-coalescent involving simultaneous multiple collisions can be found in Birkner  et al. \cite{BBMST}, where the lookdown particle representation is presented, and the
pathwise convergence of the empirical measures of approximating particle systems to the limiting $\Xi$-Fleming-Viot process is also proved.


The support properties of measure-valued Markov processes have been studied extensively. However, the majority of the literature concerns Dawson-Watanabe superprocesses, which, roughly speaking, are the spatial analogues of continuous state branching processes. They are infinitely divisible and their Laplace functionals have a useful representation in terms of solutions to semi-linear partial differential equations. These tools and others (for example the historical process and the Brownian snake) have led to very precise results concerning the supports of these processes, especially for super-Brownian motion. See \cite{P2002} for an introduction to Dawson-Watanabe superprocesses and a survey of some classical results concerning their supports.

The literature on support properties of Fleming-Viot processes is sparser. Compact support and other properties for the classical Fleming-Viot process with Brownian mutation were established in a pioneering work of Dawson and Hochberg \cite{DH82} by introducing an infinite particle system representation. A refinement of this representation, the so-called lookdown particle construction, was proposed in \cite{DK96, DK99} for Fleming-Viot processes and other measure-valued processes. The lookdown representation encodes a genealogy of  the Fleming-Viot process, and  often plays the role of the historical process for superprocesses. Using the lookdown representation, more recent work of Liu and Zhou \cite{LZ2012, LZ2015} has established, among other properties, the compact support property, the one-sided modulus of continuity and Hausdorff dimension of the support process for the $\Lambda$-Fleming-Viot process with Brownian mutation with the associated  $\Lambda$-coalescent coming down from infinity. For a measure-valued process, instantaneous support propagation occurs if the (closed) support of the random measure can reach arbitrarily far away within arbitrarily small time. Instantaneous support propagation for the super-L\'{e}vy process with binary branching and jump type spatial motion was first proved by Perkins \cite{P1990} in 1990. Surprisingly, similar instantaneous support propagation for the closely related Kingman-Fleming-Viot process with jump type mutation remained an unsolved problem. The goal of this paper is to provide an answer to this problem. 
 
Our main result, Theorem~\ref{thm_main}, establishes a version of instantaneous support propagation when the mutation is a L\'evy process with jumps. It states that the support of the measure obtained by convolving the $\Lambda$-Fleming-Viot process with the jump measure of the mutation process is contained in the support of the $\Lambda$-Fleming-Viot process itself. In many cases, this implies that the closed support is $\R^d$. The analogous statement for Dawson-Watanabe superprocesses is known \cite{EP1991, LZ2008, P1990}. We show that it holds for $\Lambda$-Fleming-Viot processes in complete generality: the mutation process can be any L\'evy process and we require no conditions on the re-sampling mechanism/ancestral coalescent.




\subsection{Statement of main result} We now introduce some notation. Let $\cM_f(E)$ and $\cM_1(E)$ denote respectively the spaces of finite measures and probability measures on a Polish space $E$. When $E = \R^d$, we simply write $\cM_f$ and $\cM_1$. Let $\cB = \cB(\R^d)$ denote the space of Borel functions from $\R^d$ into $\R$, and let $\cB^+$, $\cB_b$ and $\cB^+_b$ denote, respectively, the subspaces of non-negative, bounded, and non-negative and bounded Borel functions. For $\mu \in \cM_f$ and $\phi \in \cB$, we will use the notations $\langle \phi, \mu \rangle$ and $\mu(\phi)$ interchangeably (as appropriate) to denote the integral of $\phi$ with respect to $\mu$, that is
\begin{equation}
\langle \phi, \mu \rangle := \int_{\R^d} \phi(x) \mu(dx) =: \mu(\phi). \nonumber
\end{equation}

In order to define the $\Lambda$-Fleming-Viot process and state our main result, we introduce the objects and quantities necessary to describe the re-sampling mechanism and mutation. We begin with the former. Let $\Lambda \in \cM_f([0,1])$. For $2 \leq k \leq n \in \N$, we define the rates $\lambda_{n,k}$ by
\begin{equation}
\lambda_{n,k} = \int_0^1 x^{k-2}(1-x)^{n-k} \Lambda(dx). \nonumber
\end{equation}
These are of course the merger rates for the $\Lambda$-coalescent, which as we have noted encodes the genealogy of the $\Lambda$-Fleming-Viot process. This is described in greater detail in Sections~\ref{s_coalescents} and \ref{s_lookdown}.

Now we introduce the mutation process. Let $(W_t : t \geq 0)$ be a L\'evy process on $\R^d$. We denote its law and expectation by $\bP^W_x$ and $\E^W_x$, respectively, when $W_0 = x$. By the L\'evy-Khintchine formula, for all $t>0$ and all $\xi \in \R^d$,
\begin{equation}
\log \E^W_0(e^{i \langle W_t, \xi\rangle}) = -t \Psi(\xi), \nonumber
\end{equation}
where the characteristic exponent $\Psi$ is given by
\begin{equation}
\Psi(\xi) = i \langle a, \xi \rangle + \frac 1 2 \langle \xi, Q \xi \rangle + \int_{\R^d} \left( 1 - e^{i \langle x, \xi\rangle} + i \langle x, \xi\rangle 1_{\{|x| < 1 \}}\right) \nu(dx). \nonumber
\end{equation}
In the above, $a \in \R^d$, $Q \in \R^{d \times d}$ is a symmetric positive semidefinite matrix, and $\nu$ is a $\sigma$-finite measure on $\R^d$, called the L\'evy measure of $W_t$, satisfying $\nu(\{0\}) = 0$ and $\int_{\R^d} (1 \wedge |x|^2) \nu(dx) < \infty$; see e.g. Bertoin \cite{Bertoin96} for details. The generator of $W_t$ is the integro-differential operator $A$, which satisfies
\begin{equation} \label{eq_Levy_generator}
A \phi (x) = a \cdot \nabla \phi (x)+\frac{1}{2}\nabla \cdot (Q \nabla \phi) (x)  +\int_{\R^d}\left(\phi(x+y)-\phi(x)-\nabla \phi (x) y1_{\{ |y| < 1\}}\right)\nu(dy) \nonumber
\end{equation}
for all $\phi \in C^2_b$, the space of bounded twice differentiable functions with bounded derivatives up to order two; see Schilling \cite{Schilling}.

To define the generator of the $\Lambda$-Fleming-Viot process, we introduce an appropriate class of test functions. For $n \in \N$ and $\phi_1,\dots,\phi_n \in \cB_b$, let $F_{\phi_1,\dots,\phi_n}$ denote the function on $\cM_f$ defined by
\begin{equation}
F_{\phi_1,\dots,\phi_n}(\mu) = \prod_{i=1}^n \langle \phi_i, \mu \rangle. \nonumber
\end{equation}
Let $\cG(\cM_f)$ denote the class containing all such functions for all $n \in \N$ and $\phi_i \in \cB_b$. For $F_{\phi_1,\dots,\phi_n} \in \cG(\cM_f)$, we define
\begin{equation} \label{e_generator}
\cA^{\Lambda,A} F_{\phi_1, \dots, \phi_n} (\mu) := \sum_{i = 1}^n \langle A \phi_i, \mu \rangle \prod_{j \neq i} \langle \phi_i, \mu \rangle \,\,\, + \sum_{J \subset [n] : \#J \geq 2} \lambda_{n,\# J} (F^J_{\phi_1,\dots,\phi_n}(\mu)- F_{\phi_1,\dots,\phi_n}(\mu) ), \nonumber
\end{equation}
where $\#J$ denotes the cardinality of the set $J$. In the above, for $J \subseteq [n]$, the function $F^J_{\phi_1,\dots,\phi_n}$ is equal to $F_{\phi_1',\dots,\phi_n'}$ where $\phi_i'$ is defined as follows: if $i \not \in J$, $\phi_i' = \phi_i$; if $i \in J$, then $\phi_i' = \phi_{\min J}$. That is, the functions corresponding to indices in $J$ are all replaced by the function corresponding to the lowest index in $J$. 

The $\Lambda$-Fleming-Viot process with mutation operator $A$, or the $(\Lambda,A)$-Fleming-Viot process, is the $\cM_1$-valued Markov process with generator $\cA^{\Lambda,A}$. We denote it by $(Z_t)_{t\geq 0}$ and we write $\bP^Z_\mu$ and $\E^Z_\mu$ to denote, respectively, the law and expectation associated to the process with initial measure $Z_0 = \mu \in \cM_1$. $(Z_t)_{t \geq 0}$ is a $\cM_1$-valued strong Markov process with c\`adl\`ag paths. We denote the standard (right-continuous) filtration generated by $Z_t$ by $(\cF^Z_t)_{t \geq 0}$, and will occasionally write $(\cF_t)_{t \geq 0}$. We refer to Remark 1.1 of \cite{BB09b} for a discussion of different constructions of $Z_t$ and the form of the generator on more general functions.




Let $S(\mu)$ denote the closed (topological) support of a measure $\mu \in \cM_f$. We denote the convolution of measures $\mu$ and $\nu$ by $\mu * \nu$; for $k \in \N$, the $k$-fold convolution of $\nu$ with itself is denoted $\nu^{(k)}$. Our main result is the following.

\begin{theorem} \label{thm_main}
Let $Z_t$ be a $(\Lambda,A)$-Fleming-Viot process, where $\Lambda \in \cM_f([0,1])$ and $A$ is the generator of a L\'evy process with non-degenerate L\'evy measure $\nu$.  Then for any $t>0$, with probability one,
\begin{equation}
S( \nu^{(k)} * Z_t) \subseteq S(Z_t) \quad \text{ for all } \, k \in \N. \nonumber
\end{equation}
If $S(\nu) = \R^d$, then $S(Z_t) = \R^d$ almost surely for all $t>0$.
\end{theorem}

The case in which $S(\nu) = \R^d$ includes many important examples, such as the case when the mutation process is an $\alpha$-stable process with any index $\alpha \in (0,2)$. Our theorem of course implies that $S(Z_t)$ is unbounded whenever the mutation process has jumps. In contrast, it is known that if a $\Lambda$-coalescent comes down from infinity, then under an additional mild condition on the speed of coalescing, almost surely the support for the $\Lambda$-Fleming-Viot process with Brownian mutation remains compact at all strictly positive times; see Liu and Zhou \cite{LZ2012}. 

\subsection{Further discussion.} As previously noted, the analogous statement concerning the supports of Dawson-Watanabe superprocesses with jump type spatial motion has been known to hold for some time. For binary branching, Perkins first gave a proof in \cite{P1990}, and Perkins and Evans \cite{EP1991} gave an alternate proof shortly thereafter; see also Section III.2 of Perkins \cite{P2002}. Theorem~\ref{thm_main} establishes that instantaneous support propagation holds in full generality for $\Lambda$-Fleming-Viot processes with jump type mutation. The situation is somewhat different when one has Brownian mutation, as is indicated by the compact support theorems mentioned earlier. However, Birkner and Blath pointed out in \cite{BB09b} that if the $\Lambda$-coalescent does not come down from infinity, then the $\Lambda$-Fleming-Viot process with Brownian mutation has support equal to all of $\R^d$ a.s. for all $t>0$. The results of Liu and Zhou \cite{LZ2012} imply that staying infinite is a nearly optimal condition on the $\Lambda$-coalescent in order for this to occur, although whether or not there exist $\Lambda$-coalescents which come down from infinity for which the Brownian $\Lambda$-Fleming-Viot process has unbounded support remains unknown. For super-Brownian motion, the analogous case to the ``staying infinite" regime was recently studied by Mamin and Mytnik \cite{MM2020}.


It is generally thought that Fleming-Viot processes should have similar path properties to their Dawson-Watanabe counterparts, and our result is a verification of this heuristic for the support propagation property. The heuristic is justified in part by the explicit connection between certain sub-families of the two classes of processes. Building on work of Etheridge and March \cite{EM1991}, Perkins \cite{P1991} proved that a binary-branching super-Brownian motion conditioned to have a constant total mass equal to $1$ is a classical Fleming-Viot process with Brownian mutation. Another perspective is that the super-Brownian motion normalized by its total mass is equal in law to a Fleming-Viot process with (randomly) time-inhomogeneous Brownian mutation. This result was generalized in Birkner et al. \cite{BBCEMSW05} to a similar relation between a motionless Alpha-stable branching superprocess and mutationless Fleming-Viot process that is dual to a Beta coalescent. When the Dawson-Watanabe superprocess has a spatial motion, the corresponding $\Lambda$-Fleming-Viot process has time-inhomogeneous mutation. (The one exception to this rule is when the superprocess has Neveu's branching mechanism, which is associated to the Bolthausen-Sznitman, or Beta$(1,1$) coalescent, which stays infinite. With Brownian motion/mutation, both Neveu's superprocess and the $\Lambda$-Fleming-Viot process whose ancestral coalescent is the Beta$(1,1)$ coalescent have instantaneous support propagation \cite{BB09b, FS04}.)


The explicit connection between certain $\Lambda$-Fleming-Viot processes and Dawson-Watanabe superprocesses is a strong indication that they should share path properties, provided the associated time change is well-behaved. Blath \cite{Blath09} has proposed using this connection as a means of deducing path properties of Fleming-Viot processes from known properties of the associated Dawson-Watanabe superprocess. Such an approach is necessarily limited to Fleming-Viot processes which are dual to a Beta coalescent, as the correspondence only holds in this case; see \cite{BBCEMSW05}. To prove Theorem~\ref{thm_main}, we take a different approach which is modelled after the original proof of instantaneous propagation of Perkins \cite{P1990}. Our argument has a similar structure to Perkins', but where the original proof uses the branching property we make use of an ancestral decomposition. This is achieved using a lookdown construction. Other modifications are required to accommodate the general reproduction mechanism. Because we do not rely explicitly on the representation as a normalized, time-inhomogeneous Dawson-Watanabe superprocess, we are able to give a general proof which holds for reproduction mechanisms that are associated to general $\Lambda$-coalescents instead of only Beta coalescents.

While we do not explicitly rely upon a representation of the $\Lambda$-Fleming-Viot process as a Dawson-Watanabe superprocess, our approach is still informed by the heuristic that small-time behaviour of these processes is similar. It should be pointed out that, although the theorem of Birkner et al. \cite{BBCEMSW05} only holds for the Beta coalescents, there is a weaker connection between general $\Lambda$-coalescents and continuous-state branching processes. Berestycki et al.	 \cite{BBL14} constructed a small-time coupling between the $\Lambda$-coalescent and an associated continuous-state branching process via two lookdown representations defined via coupled point processes. They used this coupling to give an alternative proof of the speed of coming down from infinity of the $\Lambda$-coalescent. 

The approach we develop here appears to be a viable alternative to the program proposed by Blath \cite{Blath09}, mentioned above, for the study of support properties of generalized Fleming-Viot processes. As noted, the difficulty of analysing Fleming-Viot processes is in part because they are not infinitely divisible, and it is infinite divisibility/the branching property which underpins many of the proofs of support properties for Dawson-Watanabe superprocesses. Our argument substitutes the ancestral representation for the branching property. This acts as an approximate version of infinite divisibility: the decomposition of the process as the sum of a large, but not arbitrarily large, collection of identically distributed and nearly independent processes. This approximation becomes better as the time scale is taken to zero and appears to be a suitable surrogate for infinite divisibility for certain arguments. It would be interesting to see if this could be used to prove for Fleming-Viot processes some of the more precise support theorems known to hold for Dawson-Watanabe superprocesses, for example the exact Hausdorff measure function results in \cite{P1988, P1989}.

We conclude by noting that instantaneous propagation as stated here can also be a useful tool in the proof of more subtle support properties. In a recent paper \cite{H2021}, the first author used the instantaneous support propagation of the Dawson-Watanabe superprocess with $\alpha$-stable motion and $(1+\beta)$-stable branching mechanism as a tool in the proof of new properties concerning the behaviour of the density over fractal sets.

\subsection{Organization of the paper.} The remainder of the paper is organized as follows. Section~\ref{s_coalescents} defines the $\Lambda$-coalescents and states some results which we will use. In Section~\ref{s_lookdown} we introduce the lookdown construction for the $\Lambda$-Fleming-Viot process with mutation and state the properties of this model relevant to our proof. We prove Theorem~\ref{thm_main} in Section~\ref{s_pf}.

\section{$\Lambda$-coalescents} \label{s_coalescents}
We now give an overview of $\Lambda$-coalescents and state several results which we will use. There are several distinct (but related) classes of processes called ``coalescents" in probability theory; in our setting, a coalescent is an exchangeable integer partition-valued Markov process in which the only allowable transitions involve the merging, or coalescence, of blocks, and hence the partition can only become coarser. Such coalescents are models for stochastic coagulation and have applications in numerous areas. For a broad introduction to models of stochastic coagulation and fragmentation, see the book of Bertoin \cite{Bertoin_coag}. A detailed account of the theory of random exchangeable integer partitions and the $\Lambda$-coalescents can be found in the notes of Berestycki \cite{Berestycki_notes}. The $\Lambda$-coalescents form a general class of coalescing partition-valued Markov processes in which only one subset of blocks merges at any given time. Coalescents wherein multiple subsets of blocks can merge at the same time have also been studied and are referred to as $\Xi$-coalescents; see, for example Schweinsberg \cite{Schw2000a}.

For $n \in \N$, let $[n] := \{1,2,\dots,n\}$, and write $\cP_n$ and $\cP_\infty$ to denote the sets of partitions of $[n]$ and $\N$, respectively. We refer to the elements of a partition as its blocks. Let $\{\lambda_{b,k} : 1 \leq k \leq b \in \N\}$ be an array of non-negative rates. The rate $\lambda_{b,k}$ is the rate at which a given subset of $k$ blocks out of a total of $b$ blocks merges. Because a partition in $\cP_\infty$ may have infinitely many blocks, one defines the dynamics on $\cP_\infty$ via projections onto $\cP_n$ for all $n \in \N$, with the requirement that the projections are consistent and the projected process on $\cP_n$ is itself a Markov process. A structure theorem due to Pitman \cite{P1999} states that this is satisfied if and only if the array of rates satisfies
\begin{equation}
\lambda_{b,k} = \int_0^1 x^{k - 2} (1-x)^{b-k} \Lambda(dx) \nonumber
\end{equation}
for all $2 \leq k \leq b \in \N$ for some Borel measure $\Lambda \in \cM_f([0,1])$. Given $\Lambda$ and its corresponding array of rates, the $\Lambda$-coalescent, denoted $(\Pi(t) : t \geq 0)$, is the Markov process on $\cP_\infty$ such that for all $n \in \N$, its projection onto $\cP_n$, denoted $\Pi_n(t)$, is a Markov process governed by the rates $\{ \lambda_{b,k} : 2 \leq k \leq b \in [n]\}$ as described above. We write $\bP^\Pi$ to denote the law of $\Pi(t)$ and unless otherwise noted we always assume that $\Pi(0)$ is the singleton partition of $\N$. The partition $\Pi(t)$ is easily seen to be exchangeable because the merger rates do not depend on which integers are in the blocks.

Kingman's coalescent, in which all mergers are binary, corresponds to the measure $\Lambda = \delta_0$. The well-studied Beta coalescents arise when $\Lambda$ is chosen to be the distribution of a Beta$(2-\alpha,\alpha)$ random variable for $\alpha \in (0,2)$.

We write $\# A$ to denote the cardinality of a countable set $A$. Then $N_t := \# \Pi(t) \in \N \cup \{+ \infty\}$ is the number of blocks in $\Pi(t)$. We will write
\[ \Pi(t) = (\pi_{i}(t) : i = 1,\dots,N_t),\]
where the sequence is infinite if $N_t = \infty$, and we use the convention that the blocks are ordered by their minimum elements, so that $\min \{j : j \in \pi_i(t)\} < \min \{j : j \in \pi_{i+1}(t)\}$ for $i = 1, \dots, N_t - 1$.
For $i \leq N_t$ we define the \textit{asymptotic frequency} of $\pi_i(t)$ to be
\begin{equation}
|\pi_i(t)| := \lim_{n \to \infty} n^{-1} \#(\pi_i(t)\cap [n]). \nonumber 
\end{equation}
It follows from Kingman's theory of exchangeable partitions  \cite{Kingman} that the above limit exists for all $t>0$ almost surely. This result  was later reproved by Aldous \cite{Ald85} using de Finetti's theorem. 

The importance of the $\Lambda$-coalescents in our work is because they encode the genealogies of generalized Fleming-Viot processes. Indeed, we observe that given $\Lambda \in \cM_f([0,1])$, the rates $\lambda_{b,k}$ defined above are the same as those given in the introduction to define the generator of $Z_t$. The correspondence is more transparent in the non-spatial setting, and to illustrate it we briefly consider a $\Lambda$-Fleming-Viot process without mutation. This is the measure-valued Markov process on $\cM_1(E)$ for some compact Polish space $E$ (typically $E = [0,1]$) with generator
\begin{equation}
\cA_0^\Lambda F_{\phi_1,\dots,\phi_n}(\mu) := \sum_{J \subset [n] : \# J \geq 2} \lambda_{n,\# J}( F^J_{\phi_1,\dots,\phi_n}(\mu)-F_{\phi_1,\dots,\phi_n}(\mu)),
\end{equation}
where $F_{\phi_1,\dots,\phi_n}, F^J_{\phi_1,\dots,\phi_n} \in \cG(\cM_f(E))$ are as defined in the introduction for $n \in \N$, $\phi_i \in \cB_b(E)$ for $i \in [n]$, and $J \subset [n]$. The process generated by $\cA_0^\Lambda$ is a measure-valued dual process to the $\Lambda$-coalescent in the following way: one can consider the \textit{ranked mass coalescent} associated to $\Pi(t)$, which is the vector of asymptotic frequencies of blocks in $\Pi(t)$ listed in decreasing order. (Special care must be taken if $\Pi(t)$ contains singleton blocks, in which case we say $\Pi(t)$ has dust.) Provided the initial state of the mutationless $\Lambda$-Fleming-Viot process has no atoms, the process of ranked vectors of its atoms' masses is equal in distribution to the ranked mass coalescent associated to $(\Pi(t) : t \geq 0)$, which follows from the lookdown representation. See Theorem 3.1 of \cite{BBCEMSW05} for a related result.

Unlike Dawson-Watanabe superprocesses, there is no useful expression for the Laplace functional of Fleming-Viot processes. The study of Fleming-Viot processes often relies on analysis of the moments, which can be expressed using a dual process involving the coalescent.


We now resume our discussion of the number of blocks. A $\Lambda$-coalescent {\it comes down from infinity} if, when the initial partition has infinitely many blocks, $\Pi(t)$ has finitely many blocks a.s. for all $t>0$. On the other hand, it {\it stays infinite} if, with probability one, $\Pi(t)$ has infinitely many blocks for all $t>0$. Given $\Lambda(\{1\})=0$, the $\Lambda$-coalescent either comes down from infinity or stays infinite; see \cite{P1999}. Kingman's coalescent and the Beta coalescent with $\alpha \in (1,2)$ come down from infinity, while the Beta coalescents with $\alpha \in (0,1]$ stay infinite. A necessary and sufficient condition on the merger rates for $\Pi(t)$ to come down from infinity was first given by Schweinsberg \cite{Schw2000}. We review a different but equivalent criterion below.

Let $\Lambda \in \cM_f([0,1])$. We assume throughout that $\Lambda(\{1\}) = 0$. We define $\psi_\Lambda : \R^+ \to \R$ by
\begin{equation} \label{psi_def}
\psi_\Lambda(u) = \Lambda(\{0\})u^2 +
 \int_{(0,1]} \left(e^{-u x} -1 + ux \right) x^{-2} \Lambda(dx).
\end{equation}
Then $\psi_\Lambda$ is the Laplace exponent of a spectrally positive (one-dimensional) L\'evy process and hence is also the branching mechanism of a continuous state branching process (CSBP). We refer to Chapter 12 of Kyprianou \cite{Ky14} and Li \cite{Li19} for introductions on continuous-state branching and the Lamperti transform which maps between spectrally positive L\'evy processes and CSBPs.

It was observed by Bertoin and Le Gall \cite{BG2006} that the $\Lambda$-coalescent comes down from infinity if and only if
\begin{equation}
\int_1^\infty \psi_\Lambda(u)^{-1} du < \infty. \nonumber
\end{equation}
One can in fact also obtain information about the speed of coming down from infinity (i.e. the asymptotic behaviour of $N_t$ as $t\downarrow 0$) by analysing $\psi_\Lambda(u)$. For $t>0$, we define $v(t)$ by
\begin{equation}\label{def_v}
v(t) := \inf \{r > 0  : \int_r^\infty \psi_\Lambda(u)^{-1} du > t \},
\end{equation}
with the convention that $\inf \emptyset = \infty$. Then either $v(t) = \infty$ for all $t$, or $v(t)$ is finite for all $t$ and satisfies
\begin{equation}
\int_{v(t)}^\infty \psi_\Lambda(u)^{-1} du = t. \nonumber
\end{equation}
A convergence in probability version of the following result is originally due to Bertoin and Le Gall \cite{BG2006}. The version we state now was proved by Berestycki et al. \cite{BBL10}.

\begin{theorem}\label{theorem_blocks} Let $\Lambda \in \cM_f([0,1])$ be such that the $\Lambda$-coalescent comes down from infinity. Then
\begin{equation}
\lim_{t\to 0}\frac{N_t}{v(t)} = 1 \nonumber
\end{equation}
almost surely and in $L^p$ for all $p \geq 1$.\end{theorem}

A corollary of this result is a global bound on the speed of coming down from infinity. It asserts, essentially, that Kingman's coalescent comes down from infinity the fastest. Let $\Lambda \in \cM_F([0,1])$. From \eqref{psi_def}, we deduce that
\[ \psi_\Lambda(u) \leq (\Lambda(\{0\}) + \frac 1 2 \Lambda((0,1])) u^2.\]
It then follows from the definition of $v(t)$ that there is a constant $c_\Lambda > 0$ such that
\begin{equation} \label{e_blocknum_bd}
v(t) \geq c_\Lambda t^{-1} \quad \text{ for all } t>0.
\end{equation}
Theorem~\ref{theorem_blocks} then implies that for all $\eps \in (0,c_\Lambda)$,
\begin{equation}
\liminf_{t\to 0}\frac{N_t}{t^{-1}} \geq c_\Lambda - \eps \nonumber
\end{equation}
almost surely.

The proof of our main result does not require detailed information on the $\Lambda$-coalescent associated to the Fleming-Viot process. The argument is structured to allow maximum flexibility (and in fact complete generality) of the reproduction mechanism, and hence of the ancestral coalescent. The fact that $N_t$ is at least of order $t^{-1}$ for small $t$ is essentially the only property we use, and is used to prove a mild lemma (Lemma~\ref{lemma_specialsequence}).

\section{The lookdown construction} \label{s_lookdown}

In this section we introduce the countable representation, or ``lookdown" construction, of the $\Lambda$-Fleming Viot process.  The idea of using empirical measures of particle systems to approximate the Fleming-Viot process first appeared in Dawson and Hochberg \cite{DH82} to study the support of the Fleming-Viot process. Donnelly and Kurtz \cite{DK96} first proposed a  lookdown representation that provides nested exchangeable particle approximations to the classical Fleming-Viot process, and further extended the representation to more general measure-valued processes \cite{DK99}. Such a particle representation encodes the genealogy of the measure-valued process. The version of the lookdown construction used here was originally termed the ``modified" lookdown construction (see \cite{DK99}) but is now a standard version, see Birkner and Blath \cite{BB09a} and Blath \cite{Blath09}, and in the sequel we omit the qualifer ``modified."

Let $\Lambda \in \cM_F([0,1])$. We will assume that $\Lambda(\{1\}) = 0$. We consider a $(\R^d)^\infty$-valued process $X(t) = (X_1(t), X_2(t), \dots)$. For $n \in \N$, we define the empirical measure
\begin{equation}
Z^{(n)}_t := n^{-1} \sum_{i = 1}^n \delta_{X_t^i}. \nonumber
\end{equation}
If the vector $(X_1(t), X_2(t), \dots)$ is exchangeable, then by de Finetti's Theorem the empirical measures converge weakly. We can then define
\begin{equation}
Z_t := \lim_{n \to \infty} n^{-1} Z^{(n)}_t. \nonumber
\end{equation}
The lookdown construction is defined in such a way that if the initial vector $(X_1(0), X_2(0),\dots)$ is exchangeable, then so is $(X_1(t), X_2(t), \dots)$, and $Z_t$ is defined for all $t>0$.

We next follow Birkner and Blath \cite{BB09a} and Birkner et al. \cite{BBMST} to present details of the lookdown construction. Let $(X_1(0), X_2(0),\dots) \in (\R^d)^\infty$ be an exchangeable random vector. In order to define the process $(X_1(t), X_2(t), \dots)$, we introduce several families of Poisson point processes. First, let $\{\bN_{ij}(t) : 1 \leq i < j \in \N \}$ be a family of independent Poisson processes on $[0,\infty)$ with rate $\Lambda(\{0\})$. At each jump time of the process $\bN_{ij}$, the particle at level $j$ looks down (since $j >i$) to the particle at level $i$ and takes its type (or location). A jump of $N_{ij}$ therefore corresponds to a reproduction event for the particle at level $i$. The levels of other particles are shifted up to account for the birth. That is, supposing that $\Delta \bN_{ij}(t) = 1$, we have
\begin{equation*}
X_k(t) =
\begin{cases}  X_k(t-) & \text{ if } k < j,\\
X_i(t-) & \text{ if } j = k,\\
X_{k-1}(t-) & \text{ if } j < k. \end{cases}
\end{equation*} 
The dynamics described above govern the binary mergers and hence the ``Kingman" component of the $\Lambda$-coalescent embedded in the model. We now describe the dynamics arising from multiple mergers. We define $\Lambda_0 = \Lambda - \Lambda(\{0\})\delta_0$. Let $\bN$ be a Poisson point process on $[0,\infty) \times (0,1]$ with intensity $dt \otimes u^{-2}\Lambda_0(du)$. Let $\{(t_i, u_i) : i = 1,2,\dots \}$ be an enumeration of the points in $\bN$. For each point $(t_i, u_i)$, we associate a sequence $\{Y_{ik} : k=1,2,\dots \}$ of independent Bernoulli$(u_i)$ random variables. If $Y_{ik} = 1$, then the particle at level $k$ at time $t_i$ takes part in the birth event at this time. All participating levels take the location of the lowest participating level, and all the other particle positions retain their original order and are shifted upward accordingly. If the jump occurs at time $t = t_i$, and $j$ is the lowest index such that $Y_{ij} = 1$, then
\begin{equation}
X_k(t) = \begin{cases}
X_k(t-) & \text{ if } k \leq j, \\
X_j(t-) & \text{ if } k > j\text{ and }  Y_{ik} = 1 , \\
X_{k - J_{ik}} (t-) & \text{ if } k > j \text{ and }  Y_{ik} = 0, \end{cases} \nonumber
\end{equation}
where $J_{ik} = \# \{ l < k : Y_{il} = 1\} - 1$.

Let $(W_1(t), W_2(t), \dots)$ be a sequence of independent mutation processes. Between the jump times, the increments of $X_i(t)$ are those of $W_i(t)$ for each $i = 1,2,\dots$. These dynamics are well-defined, as one can show that for any $i \in \N$, the number of lookdown events involving level $i$ up to time $t$ is finite. One can rigorously realize the dynamics of $(X_1(t),X_2(t),\dots)$ as a countable system of stochastic differential equations driven by $\{\mathbf{N}_{ij} : 1 \leq i < j \in \N \}$, $\mathbf{N}$, $\{Y_{ik} : i, k \in \N\}$ and $(W_1(t) , W_2(t),\dots)$. Rather than include this here, we refer to \cite{DK99, BBMST} for details. 

Hereafter, we will always choose to realize the lookdown process by choosing the initial particle positions of $(X_1(0), X_2(0), \dots)$ to be i.i.d. samples from a probability measure $\mu \in \cM_1$.  We denote by $\cL^{A,\Lambda}_\mu$ a probability measure which realizes the above construction with this choice of initial positions. The following is a consequence of Theorems 1.1 and 3.2 of \cite{DK99}.

\begin{theorem}
Under $\cL^{A,\Lambda}_\mu$, $(X_1(t),X_2(t),\dots)$ is exchangeable for all $t>0$, and
\begin{equation}
Z_t = \lim_{n \to \infty} n^{-1} \sum_{i=1}^n \delta_{X_i(t)} \nonumber
\end{equation}
exists in $\cM_1(\R^d)$ for all $t>0$ almost surely. Furthermore,
\begin{equation}
\cL^{A,\Lambda}_\mu(Z_\cdot \in \cdot) = \bP^Z_\mu(Z_\cdot \in \cdot). \nonumber
\end{equation}
\end{theorem}

Thus we can realize the $\Lambda$-Fleming-Viot process with mutation as the process of (limiting) empirical measures of the particle system under $\cL^{A,\Lambda}_\mu$. As we now discuss, this construction of $Z_t$ encodes its genealogy.

Fix $t>0$ and consider the process $(X_1(s),X_2(s),\dots)$ running forward to time $s$. We consider time in reverse, backwards from $t$. For the individual which is at level $i$ at time $t$, we denote the level of its ancestor at time $t-s$ by $L^t_i(s)$, for $s \in [0,t]$. (Note that $L^t_i(0) = i$ for all $i \in \N$.) For fixed $t>0$, the collection of ancestor processes, $(L^t_i(s) : s \in [0,t])$ for $i \in \N$, satisfy a natural system of stochastic diferential equations driven by the lookdown construction's Poisson point processes; see Section 5 of \cite{DK99}. We emphasize that $s \in [0,t]$ is the time \textit{before} the reference time $t$; given $0 \leq s \leq t$, the ancestor level $L^t_i(s)$ corresponds to an individual in the population at (forward) time $t-s$. This is the convention used e.g. by Birkner and Blath \cite{BB09b}.

To recover the ancestral coalescent, we define blocks which consist of all individuals with a given common ancestor. With $t$ fixed as above and $s \in [0,t]$, the blocks in the ancestral partition at time $s$ before $t$ are the classes defined by the equivalence relation $i \sim j$ if $L^t_i(s) = L^t_j(s)$. The collection of such equivalence classes is an integer partition and we denote it by $\Pi^t(s)$. (Again, our convention is that $s$ denotes the time before $t$ and thus $\Pi^t(s)$ corresponds to the forward time process at time $t-s$.)

\begin{proposition} Given $t>0$, $(\Pi^t(s) : s \in [0,t])$ has the law of a $\Lambda$-coalescent started at the singleton partition running to time $t$.
\end{proposition}

Thus the ancestral coalescent of the $\Lambda$-Fleming-Viot process is a $\Lambda$-coalescent. We define the number of blocks
\begin{equation}
N^t_s:= \# \Pi^t(s)
\end{equation}
and write
\begin{equation}
\Pi^t(s) = (\pi^t_1(s), \pi^t_2(s), \dots, \pi^t_{N^t_s}(s)),
\end{equation}
where the blocks are again ordered by their minimum elements and the sequence is infinite if $N^t_s = + \infty$. The following lemma is proved in \cite{LZ2012}.

\begin{lemma} \label{lemma_labels}
Given $0 \leq s \leq t$, for all $i = 1,\dots, N^t_s$, it holds that $L^t_j(s) = i$ for all $j \in \pi^t_i(s)$.
\end{lemma}

In other words, the label of a block is equal to the level of its ancestor. Next, we recall that each block has an asymptotic frequency
\begin{equation}
|\pi^t_i(s)| := \lim_{n \to \infty} n^{-1} \sum_{j=1}^n 1(j \in \pi^t_i(s)). \nonumber
\end{equation}
For $\phi \in \cB_b(\R^d)$, $n \geq 1$ and $0 \leq s \leq t$, we define
\begin{equation} \label{def_n_cluster_measure}
Z_{i,s}^{(n)}(t,\phi) = n^{-1} \sum_{j=1}^n 1(j \in \pi^t_i(s)) \phi(X_j(t))
\end{equation}
and
\begin{equation} \label{def_cluster_measure}
Z_{i,s}(t,\phi) = \liminf_{n \to \infty} Z_{i,s}^{(n)}(t,\phi).
\end{equation}
If $\phi \equiv 1$, then the limit exists and is equal to the asymptotic frequency $|\pi^t_i(s)|$. If $\phi(x) = 1_B(x)$ for a Borel set $B \subseteq \R^d$, we will denote the above quantity by $Z_{i,s}(t,B)$. Intuitively, $Z_{i,s}(t,\cdot)$ is the measure associated to the cluster at time $t$ whose common ancestor at time $s$ before $t$ is the level $i$ individual.

We end the section by introducing several $\sigma$-algebras associated to the lookdown model. To encode the behaviour of the Fleming-Viot process and ancestral coalescent under $\cL_\mu^{A,\Lambda}$, we can use either the natural filtrations associated to $Z_t$ and $\Pi^T(t)$ or the filtrations generated by the auxiliary processes, that is, the point and mutation processes. Because our needs are relatively modest, we opt for the former.

Let $\cF^Z_t = \sigma( Z_s : s \leq t)$. Then $(Z_t)_{t \geq 0}$ is a Markov process with respect to the filtration $(\cF^Z_t)_{t \geq 0}$ and we have
\begin{equation}
\cL^{A,\Lambda}_\mu(f(Z_{t + \cdot}) \, | \, \cF^Z_t) = \cL^{A,\Lambda}_{Z_t}(f(Z_\cdot)) \nonumber
\end{equation}
for suitable functions $f$.

For $0\leq s \leq t$, we define $\cF_{t,s}^\Pi = \sigma(\Pi^t(s) : s \leq t)$. $(\cF^\Pi_{t,s})_{s \leq t}$ is then the natural filtration generated by the ancestral coalescent backwards from time $T$. The $\sigma$-algebra $\cF_{t,t}^\Pi$ thus encodes the entire ancestral coalescent from time $T$ back to time $0$. For this special case we adopt the simplified notation $\cF^\Pi_t : = \cF^\Pi_{t,t}$.

We will simply write $\cL^{A,\Lambda}_\mu(\cdot\,| \, X(s))$ when conditioning on the vector $X(s)$ of particle positions at time $s$.

The proof of our main result uses the following basic properties of the lookdown representation.

\begin{lemma} \label{lemma_ancestor_law}
Given any $0\leq s \leq t$, $i, j \in \N$ with $i < j$, and $\mu \in \cM_1$, for any Borel set $B \subset \R^d$,
\begin{equation}
\cL^{A,\Lambda}_\mu( X_j(t) \in B , j \in \pi^t_i(s)\, | \, \cF^\Pi_{t,s}, X(t-s)) = \bP^W_{X_i(t-s)}(W_s \in B) 1(j \in\pi^t_i(s)) \nonumber
\end{equation}
\end{lemma}

\begin{lemma}  \label{lemma_indep_ancestors}
For any $0 < s \leq t$ and $i, j \in \N$ with $i \neq j$, $Z_{i,s}(t,\cdot)$ and $Z_{j,s}(t,\cdot)$ are conditionally independent given $\cF^\Pi_{t,s} \vee \cF^Z_{t-s}$.
\end{lemma}

Both results are elementary and we omit the proofs.



\section{Proof of main result} \label{s_pf}
We begin with the statements and proofs of several lemmas. First, we fix notation which will be in effect for the remainder of the paper. We continue to write $(W_s)_{s\geq 0}$ to denote a copy of the mutation process, i.e. the L\'evy process with generator $A$ given by \eqref{eq_Levy_generator}. Let $(T_s)_{s > 0}$ denote the corresponding semigroup, which is the strongly continuous contraction semigroup defined as
\begin{equation}
T_s\phi(x) = \E^W_x(\phi(W_s)) \nonumber
\end{equation}
for $x \in \R^d$ and $s>0$.

Let $\Lambda \in \cM_f([0,1])$. We define
\begin{equation}
\sigma = \sigma(\Lambda) = \Lambda([0,1]). \nonumber
\end{equation}
We note that $\sigma = \lambda_{2,2}$, the rate at which the blocks containing a given pair of distinct indices merge in the $\Lambda$-coalescent.

Let $(Z_t)_{t \geq 0}$ be the Fleming-Viot process with mutation process $(W_s)_{s \geq 0}$ and ancestral coalescent encoded by $\Lambda$. 
The following lemma contains standard first and second moment formulae for $Z_t$.

\begin{lemma} \label{lemma_FV_moments}Let $\mu \in \cM_1$.
\noindent (a) For $\phi \in \cB_b$,
\begin{equation}
\E^Z_\mu(\langle \phi, Z_t \rangle) = \langle T_t \phi, \mu \rangle. \nonumber
\end{equation}
\noindent (b) For $\phi, \psi \in \cB_b$,
\begin{align}
\E^Z_\mu(\langle \phi, Z_t \rangle \langle \psi, Z_t \rangle) = \exp(-\sigma t)\langle T_t \phi, \mu \rangle \langle T_t \psi, \mu \rangle + \Big\langle \int_0^t \sigma \exp(-\sigma s)) T_{t-s}(T_s\phi T_s \psi) ds, \mu \Big\rangle. \nonumber
\end{align}
\end{lemma}
We omit the proof. These formulae (as well as those for higher moments) can be seen by constructing a coalescing function-valued dual process to $Z_t$. See, for example, Sections~1.12 and 2.8 of the book of Etheridge \cite{Eth00}, where the argument is written in full for the case $\Lambda = \delta_0$.

Suppose that $\phi \in \cB_b^+$ with $\phi(x) \leq M$ for all $x$ for some $M>0$. Then Lemma~\ref{lemma_FV_moments}(b) implies that
\begin{equation} \label{e_FV2momentbd}
\E^Z_\mu(|\langle \phi, Z_t \rangle|^2) \leq \exp(-\sigma t)|\langle T_t \phi, \mu \rangle|^2  + M(1-\exp(-\sigma t)) \langle T_t \phi, \mu \rangle.
\end{equation}

The series of lemmas which follow use the lookdown representation of $Z_t$ to bound below the probability that the mass of $Z_t$ on a given Borel set exceeds some positive value. We consider a general Borel set $B$, but the reader may think of $B$ as being an open ball. 

For the remainder of the work, we will denote the lookdown measure $\cL^{A,\Lambda}_\mu$ simply by $\cL_\mu$. Unless otherwise stated there are no additional assumptions on either the generator $A$ or the measure $\Lambda$.

We recall the definitions \eqref{def_n_cluster_measure} and \eqref{def_cluster_measure} of the cluster measures $Z_{i,s}^{(n)}(t,\cdot)$ and $Z_{i,s}(t,\cdot)$, for $i \in [N^t_s]$, consisting of the individuals in the coalescent block $\pi^t_i(s)$. Recall also from Lemma~\ref{lemma_labels} that the ancestor of $\pi^t_i(s)$ is the individual on level $i$. In the interest of simplifying our statements, hereafter it is implicit that $i \in [N^t_s]$ whenever we discuss $\pi^t_i(s)$, and hence $\pi^t_i(s)$ is a block in the partition $\Pi^t(s)$.

\begin{lemma} \label{lemma_local_cluster_mass}
For any Borel set $B \subseteq \R^d$, $0 < s \leq t$, and $\mu \in \cM_1$, we have
\begin{equation}
\cL_\mu \left(Z_{i,s}(t,B) \geq \frac{p|\pi_i^t(s)|}{2} \, \bigg|\, \cF^\Pi_{t,s}, X(t-s) \right) \geq \frac p 2, \nonumber
\end{equation}
where $p = \bP^W_{X_i(t-s)}(W_s \in B)$.
\end{lemma}


\begin{proof}
We first observe that $p$ is measurable with respect to $X(t-s)$ and so the statement makes sense. Without loss of generality we assume that $|\pi_i(t)| > 0$, as the result holds trivially if $|\pi^t_i(s)| = 0$. By Lemma~\ref{lemma_ancestor_law}, for $j \in \N$ (and assuming $i \in [N^t_s]$),
\begin{equation}
\cL_\mu (X_j(t) \in B, j \in \pi^t_i(s) | \cF^\Pi_{t,s}, X(t-s)) = \bP^W_{X_i(t-s)}(W_{s} \in B) 1(j \in \pi^t_i(s)). \nonumber
\end{equation}
It follows that
\begin{equation}
\cL_\mu (Z^{(n)}_{i,s}(t,B) \, | \, ) = \bP^W_{X_i(t-s)}(W_s \in B) \, Z^{(n)}_{i,s}(t,1). \nonumber
\end{equation}
Taking the limit infimum of both sides, and applying Fatou's Lemma and the fact that $Z^{(n)}_i(t,1)$ converges to $\pi_i(t)$, we obtain that
\begin{equation} \label{e_localmassaux1}
\cL_\mu (Z_{i,s}(t,B) \, | \, \cF^\Pi_{t,s}, X(t-s)) \geq \bP^W_{X_i(t-s)}(W_s \in B) |\pi^t_i(s)|.
\end{equation}
Let $E = \{Z_{i,s}(t,B) \geq p |\pi^t_i(s)|/2\}$. Then
\begin{equation}
Z_{i,s}(t,B) \leq \frac{p |\pi^t_i(s)|}{2} 1_{E^c} + |\pi^t_i(s)| 1_{E}. \nonumber
\end{equation}
Taking the conditional expectation of both sides of the above and applying \eqref{e_localmassaux1} gives
\begin{equation}
p |\pi_i^t(s)| \leq  \frac{p |\pi^t_i(s)|}{2}(1 - \cL_\mu(E \, | \, \cF^\Pi_{t,s}, X(t-s)))+ |\pi^t_i(s)|\cL_\mu (E \, | \, \cF^\Pi_{t,s}, X(t-s)). \nonumber
\end{equation}
Rearranging the above, we have
\begin{equation}
\cL_\mu(E \, | \, \cF^\Pi_{t,s}, X(t-s)) \geq \frac{p}{
2}\left(1 - \frac p 2 \right)^{-1} \geq \frac p 2, \nonumber
\end{equation}
which proves the result.
\end{proof}

For $B \subset \R^d$ and $\epsilon > 0$, we define the $\epsilon$-enlargement $B_\epsilon := \{x \in \R^d : d(x,B) < \epsilon\}$, where $d(x,B) := \inf_{y \in B} |x-y|$. Equivalently, $B_\epsilon = \{x + y : x \in B, y \in B(0,\epsilon)\}$, so $B_\eps$ is clearly open. For $t,\epsilon>0$, we define
\begin{equation} \label{def_pdeltaeps}
p(t,\epsilon) = \bP^W_0(|W_t|< \epsilon).
\end{equation}

To state the next lemma, we introduce some notation. If $(\Pi(s))_{s \geq 0}$ is a $\Lambda$-coalescent, for $s>0$ and $b \in [0,1]$, we denote by $N_s(b)$ the number of blocks in $\Pi(s)$ with asymptotic frequency greater than or equal to $b$. That is, recalling that $\Pi(s) = (\pi_1(s), \dots, \pi_{N_s}(s))$,
\begin{equation}
N_t(b) = \# \{ i \in [N_t] : |\pi_i(t)| \geq b\}. \nonumber
\end{equation}
Note that $N_t(0) = N_t$, the number of blocks in the coalescent. If $(\Pi^t(s))_{s \in [0,t]}$ is the ancestral coalescent backwards from time $t$ in the lookdown construction, we denote the analogous quantity by $N^t_s(b)$. That is, for $0 \leq s \leq t$, 
\begin{equation}
N^t_s(b) := \# \{ i \in [N^t_s] : |\pi_i^t(s)| \geq b \}. \nonumber
\end{equation}


\begin{lemma} \label{lemma_clusterestimate}
For any $\mu \in \cM_1$, $t, \epsilon > 0$, $b \in (0,1/2]$ and Borel $B \subset \R^d$, we have
\begin{equation}
\cL_\mu ( Z(t, B_\epsilon) \geq  b\cdot p(t,\epsilon)\, | \, \cF^\Pi_t) \geq 1 - \left(1 -  \frac{\mu(B) \cdot p(t,\epsilon)}{2} \right)^{N^t_t(2b)}
\end{equation}
\end{lemma}

\begin{proof}
We define the event
\begin{equation}
E = \{Z(t,B_\epsilon) \geq b \cdot p(t,\epsilon)\}.
\end{equation}
We condition on $\cF^\Pi_t$, with respect to which $N^t_t(2b)$ is obviously measurable. To simplify notation we will write $N = N^t_t(2b)$. Suppose that the $N$ blocks of $\Pi^t(t)$ with mass at least $2b$ are $\{\pi^t_{i_1}(t),\dots,\pi^t_{i_N}(t)\}$ for some indices $i_1, \dots, i_N$. We hereafter omit the time parameter and write $\pi_{i_j}$ for $j\in [N]$. For each $j \in [N]$, let $x_{i_j} = X_{i_j}(0)$, that is, $x_{i_j}$ is the location of the ancestor of $\pi_{i_j}$, and define the events
\begin{equation}
E^1_j = \{x_{i_j} \in B\} \nonumber
\end{equation}
and
\begin{equation}
E^2_j = \{Z_{i_j,t}(t,B(x_{i_j},\epsilon)) \geq p(t,\epsilon) |\pi_{i_j}| / 2 \}, \nonumber
\end{equation}
where $B(x_{i_j},\epsilon)$ is the open ball of radius $\epsilon$ centered at $x_{i_j}$. The events $E^1_1,\dots,E^1_N, E^2_1,\dots, E^2_N$ are pairwise independent given $\cF^\Pi_t$. To see this, we first note that $E^1_1,\dots,E^1_N$ are independent because the vector of initial particle positions (from which $x_{i_j}$ are taken) is a vector of i.i.d. samples from $\mu$. By translation invariance of the spatial motions, the event $E^2_j$ is independent of the value of $x_{i_j}$, and hence of $E^1_j$, and it is trivially independent of $E^1_{j'}$ for $j' \neq j$. Finally, $E^2_1,\dots,E^2_N$ are independent by Lemma~\ref{lemma_indep_ancestors}.

Next, we argue that if $E_j^1 \cap E_j^2$ occurs for any $j \in [N]$, then $E$ occurs. To see this, we first note that on $E_j^1$ we have $B(x_{i_j}, \epsilon) \subseteq B_\epsilon$. Hence, if $E^2_j$ occurs as well, we have
\begin{align}
Z_{i_j, t}(t, B_\epsilon) \geq p(t,\epsilon) |\pi_{i_j}|/2 \geq p(t,\epsilon) b.  \nonumber
\end{align}
The claim then follows because $Z(t,B_\epsilon) \geq Z_{i_j,t}(t,B_\epsilon)$. We therefore obtain that
\begin{align} \label{e_clusterlemma1}
\cL_\mu (E \, | \, \cF^{\Pi}_t) &\geq \cL_\mu ( \cup_{j=1}^N E^1_j \cap E^2_j \, | \, \cF^\Pi_t ) \nonumber
\\ &= 1 - \prod_{j=1}^N (1 - \cL_\mu(E^1_j \cap E^2_j \, | \, \cF^\Pi_t)) \nonumber
\\ &= 1 - (1- \cL_\mu (E^1_1 \cap E^2_1 \, | \, \cF^\Pi_t))^N,
\end{align}
where the second line uses conditional independence and the third line follows because the events have the same probability for each $j \in [N]$. (In the case of $E^2_j$, this is due to translation invariance of the spatial motion.) The point $x_{i_1}$ is a random variable with distribution $\mu$, and hence $\cL_\mu (E^1_1 \, | \, \cF^\Pi_t) = \mu(B)$. To bound the probability of $E^2_1$ below, we apply Lemma~\ref{lemma_local_cluster_mass} directly. By translation invariance, this yields
$\cL_\mu(E^2_1 \, | \, \cF^\Pi_t) \geq p(t,\epsilon)/2$, where we recall the definition of $p(t,\eps)$ from \eqref{def_pdeltaeps}. Applying independence again and using these bounds, we obtain from \eqref{e_clusterlemma1} that
\begin{align}
\cL_\mu(E \, | \, \cF^\Pi_t) &\geq 1 - \left(1 - \frac{ \mu(B) p(t, \epsilon)}{2}\right)^N, \nonumber
\end{align}
completing the proof.
\end{proof}




\begin{lemma} \label{lemma_masslwrbd}
Given any $0 < \delta < t_0$, $\mu \in \cM_1$ and $B \subseteq \R^d$, define $m = T_{t_0-\delta} \mu (B)$. Then for any $\eps > 0$ and $b \in (0,1/2]$, we have
\begin{equation} \label{e_lemma_masslwrdbd}
\bP^Z_\mu \left( Z(t_0,B_\epsilon) \geq b \cdot p(\delta,\epsilon) \right) \geq \frac{\E^\Pi(1 - e^{-\theta})}{8(\sigma+1)} \left[ 1 \wedge \frac{m}{ t_0-\delta} \right],
\end{equation}
where $\theta = m + \sigma (t_0-\delta) p(\delta,\epsilon) N_\delta(2b)$, and $\Pi$ is the ancestral coalescent of $Z$.
\end{lemma}

\begin{proof}
We begin with the conditional probability of the desired event given $\cF_{t_0 -\delta}$ and apply the Markov property at time $t_0 -\delta$, yielding
\begin{align} \label{e_masslemmaaux00}
\bP^Z_\mu(Z(t_0,B_\eps) \geq b \cdot p(\delta,\eps) \, | \, \cF_{t_0 - \delta}) &= \E^Z_{Z_{t_0 - \delta}}  (Z(\delta,\Sigma_\eps) \geq b \cdot p(\delta,\eps))\nonumber
\\ &= \cL_{Z_{t_0 - \delta}}(Z(\delta,B_\eps) \geq b \cdot p(\delta,\eps)).
\end{align}
The second line simply realizes the Fleming-Viot process via a lookdown construction. We estimate the quantity above by conditioning on the ancestral coalescent.

Let $\tilde{\mu} \in \cM_1$ and let $Y$ denote a version of the conditional probability of $\{Z(\delta,B_\eps) \geq b\cdot p(\delta,\eps)\}$ given $\cF^\Pi_\delta$ with respect to the measure $\cL_{\tilde \mu}$. Then by Lemma~\ref{lemma_clusterestimate} we have
\begin{equation}
Y \geq 1 - \left( 1 - \frac{\tilde{\mu}(B) \cdot p(\delta,\eps)}{2}\right)^{N^\delta_\delta(2b)}, \nonumber
\end{equation}
where we recall that $N_\delta^\delta(2b)$ denotes the number of blocks in $\Pi^\delta(\delta)$ with asymptotic frequency at least $2b$. It then follows that
\begin{align}
\cL_\nu(Z(\delta,B_\eps) \geq b \cdot p(\delta,\eps)) &= \cL_{\tilde{\mu}}(Y) \nonumber
 \\ & \geq  \cL_{\tilde{\mu}} \bigg(  1 - \left( 1 - \frac{\tilde{\mu}(B) \cdot p(\delta,\eps)}{2}\right)^{N^\delta_\delta(2b)} \bigg) \nonumber
 \\ & =  \E^\Pi \bigg(1 - \left( 1 - \frac{\tilde{\mu}(B) \cdot p(\delta,\eps)}{2}\right)^{N_\delta(2b)} \bigg). \nonumber
\end{align}
In the last line, the expectation under $\cL_{\tilde{\mu}}$ can be written as an expectation with respect to $\E^\Pi$ because the only randomness is through the ancestral coalescent (via $N^\delta_\delta(2b)$). Because this holds for any $\tilde{\mu} \in \cM_1$, we may apply it in \eqref{e_masslemmaaux00} to obtain
\begin{equation}
\bP^Z_\mu(Z(t_0,B_\eps) \geq b \cdot p(\delta,\eps) \, | \, \cF_{t_0 -\delta}) \geq \E^\Pi \bigg(1 - \left( 1 - \frac{Z(t_0 - \delta,B) \cdot p(\delta,\eps)}{2}\right)^{N_\delta(2b)} \bigg). \nonumber
\end{equation}
Finally, we take the expectation of both sides under $\E^Z_\mu$ to obtain
\begin{equation}\label{e_masslemmaaux11}
\bP^Z_\mu(Z(t_0,B_\eps) \geq b \cdot p(\delta,\eps)) \geq \E^Z_\mu \otimes \E^\Pi \bigg(1 - \left( 1 - \frac{Z(t_0 - \delta,B) \cdot p(\delta,\eps)}{2}\right)^{N_\delta(2b)} \bigg),
\end{equation}
where $\E^Z_\mu \otimes \E^\Pi$ denotes the expectation with respect to the product of the two measures. To simplify notation we will hereafter write $N = N_\delta(2b)$. The elementary inequality $(1-x)^n \leq e^{-nx}$ for all $x \in [0,1]$ and $n\in \N$ implies that
\begin{equation}\label{e_masslemmaaux1}
1 -  \left(1 - \frac{ Z(t_0 -\delta, B)  p(\delta, \epsilon)}{2}\right)^N \geq 1 -  \exp \left(- \frac{N  Z(t_0 - \delta, B)  p(\delta, \epsilon)}{2}\right).
\end{equation}

The next step is the approximation $1 - e^{-x} \approx x$, but we need to keep track of constants. For any $\theta > 0$, $1 - e^{-x} \geq x \theta^{-1} (1 - e^{-\theta})$ for all $ x\in [0,\theta]$. Let $\theta>0$. By restricting to the event where the exponent has absolute value less than $\theta$, it follows that
\begin{align} \label{e_masslemmaaux2}
&1 - \exp \left(- \frac{N Z(t_0 - \delta, B) p(\delta, \epsilon)}{2}\right) \nonumber
\\ &\hspace{1 cm} \geq \frac{(1-e^{-\theta}) N p(\delta,\epsilon) }{2 \theta}   Z(t_0-\delta,B) 1\left(Z(t_0-\delta,B)  \leq \frac{2 \theta}{N p(\delta, \epsilon)}\right)
\end{align}
To obtain a lower bound on the terms involving $Z(t_0 - \delta, B)$, we proceed as follows:
\begin{align}
&Z(t_0-\delta,B) 1\left(Z(t_0-\delta,B)  \leq \frac{2 \theta}{N p(\delta, \epsilon)}\right) \nonumber
\\ &\hspace{1 cm}= Z(t_0-\delta,B)\left[1 - 1 \left(\frac{Z(t_0-\delta,B) \, N  p(\delta,\eps) }{2 \theta} > 1 \right)\right] \nonumber
\\ &\hspace{1 cm}\geq Z(t_0 -\delta,B) -  Z(t-\delta,B)^2 \frac{\, N  p(\delta,\epsilon)}{2 \theta}. \nonumber
\end{align}
From Lemma~\ref{lemma_FV_moments}(a) and the definition of $m$, we have $m = \E^Z_\mu(Z(t_0-\delta,B)) = T_{t_0-\delta} \mu (B)$. For the second moment, the estimate \eqref{e_FV2momentbd} implies
\begin{equation} \label{e_masslemma_2ndmoment}
\E^Z_\mu ( Z(t_0-\delta,B)^2) \leq e^{-\sigma(t_0 - \delta)} m^2 + (1 - e^{-\sigma(t_0 - \delta)}) m \leq m^2 + \sigma(t_0 - \delta) m.
\end{equation}
Given $N$, taking the expectation on both sides of the previous inequality, this moment estimate yields
\begin{align} \label{e_masslemmaaux3}
&\E^Z_\mu \left(Z(t_0-\delta,B) 1\left(Z(t_0-\delta,B)  \leq \frac{2 \theta}{N p(\delta,\eps)}\right)\right) \nonumber
\\ &\hspace{1 cm}\geq m - \left[e^{-\sigma (t_0-\delta)} m^2 + (1-e^{-\sigma (t_0-\delta)}) m \right] \frac{N p(\delta,\eps)}{2\theta} \nonumber
\\ &\hspace{1 cm}\geq m - m [m + \sigma(t_0-\delta)]  \frac{N p(\delta,\eps)}{2\theta}.
\end{align}
We now choose the value of $\theta$ to be
\[\theta = [m + \sigma (t_0 - \delta)]	p(\delta,\epsilon) N\]
and observe the cancellations this choice will cause in the last expression of \eqref{e_masslemmaaux2} and in \eqref{e_masslemmaaux3}. We combine \eqref{e_masslemmaaux1}, \eqref{e_masslemmaaux2} and \eqref{e_masslemmaaux3} to obtain
\begin{align}
\E^Z_\mu \otimes \E^\Pi \bigg( 1 - \left(1 - \frac{ Z(t_0 -\delta, B) p(\delta, \epsilon)}{2}\right)^{N} \bigg) &\geq  \E^\Pi \bigg( \frac{(1 - e^{-\theta})(m - m/2)}{2[m+ \sigma(t_0 - \delta)]} \bigg) \nonumber
\\ &\geq \frac{ \E^\Pi(1 - e^{-\theta})}{8(\sigma + 1)} \left[ 1 \wedge \frac{m}{t_0 - \delta}\right]. \nonumber
\end{align}
By \eqref{e_masslemmaaux11}, this completes the proof. \end{proof}


Recall that $(T_s)_{s > 0}$ is the semigroup associated to $W_s$, the Markov process with generator $A$, which has form
\begin{equation}\label{Levy_generator}
A\phi(x) =  \nabla \cdot (Q\nabla \phi)(x) + a \cdot \nabla \phi(x) +  \int_{\R^d} \left( \phi(x+y) - \phi(x) - \nabla \phi(x) y 1_{\{|y| < 1\}}\right) \nu(dy), 
\end{equation}
where $a$, $Q$ and $\nu$ are as described in the introduction. We write $T_s^*$ to denote the adjoint of $T_s$. It is standard that $(T_s^*)_{s > 0}$ is a Markov semigroup and its generator, $A^*$, generates a L\'evy process we denote $W^*_s$ which is the time reversal of $W_s$. In particular, we have
\begin{equation} \label{eq_adjointgenerator}
A^*\phi(x) = \nabla \cdot (Q \nabla \phi)(x) - b \cdot \nabla \phi(x) + \int_{\R^d} [ \phi(x-y) - \phi(x) + \nabla \phi(x) y 1_{(0,1]}(|y|)] \nu(dy).
\end{equation}
For details, see Section II.1 of Bertoin \cite{Bertoin96}.

We recall that L\'evy processes are Feller, and in particular $(T_s)_{s > 0}$ and $(T_s^*)_{s>0}$ are strongly continuous contraction semigroups on the Banach space $C_0$ of continuous functions which vanish at infinity equipped with the topology of uniform convergence. (See Proposition I.5 of \cite{Bertoin96}.)

We hereafter adopt the convention that $\infty \cdot 0 := 0$. The convolution of two measures $\mu$ and $\nu$ is denoted $\mu * \nu$.
\begin{lemma} \label{lemma_liminfgenerator}
Given any $t,\eps>0$, bounded $B \subset \R^d$, and positive, strictly decreasing sequences $(\delta_n)_{n \geq 1}$ and $(\epsilon_n)_{n \geq 1}$ converging to $0$, with probability one we have
\begin{equation}
\liminf_{n \to \infty} \eps_n^{-1} T_{\eps_n} Z_{t-\delta_n} (B_\epsilon) \geq (\nu * Z_t)(B_{\epsilon}) + \infty \cdot Z_t (B_{\epsilon}). \nonumber
\end{equation}
\end{lemma}

\begin{proof}
First we recall that $t \to Z_t$ is c\`adl\`ag in the weak topology, and hence the (weak) left limit $Z_{t-} = \lim_{s \uparrow t}Z_s$ exists for all $t>0$. Furthermore, the process has no fixed time discontinuities, so given $t>0$ we have $Z_{t-} = Z_t$ almost surely. In particular, if $U \subseteq \R^d$ is open and $t>0$ is fixed, then $\lim_{s \uparrow t} Z_s(U) = Z_t(U)$ almost surely (in the weak topology). We use this fact several times in what follows. 

Let $t,\eps,B,(\delta_n)_{n\geq 1}$ and $(\eps_n)_{n \geq 1}$ be as in the statement of the lemma. We decompose $\eps_n^{-1} T_{\eps_n} Z_{t-\delta_n} (B_\epsilon)$ by writing
\begin{align} \label{e_liminfsemigroup1}
\eps_n^{-1} T_{\eps_n} Z_{t-\delta_n} (B_\epsilon) &= \eps_n^{-1} \langle 1_{B_\eps}, T_{\eps_n} Z_{t-\delta_n} \rangle \nonumber
\\ &= \eps_n^{-1} \langle T_{\eps_n}^* 1_{B_\eps},  Z_{t-\delta_n} \rangle \nonumber
\\ &= \int_{B_\epsilon^c} \epsilon_n^{-1} T_{\epsilon_n}^* 1_{B_{\epsilon}}(x) Z_{t-\delta_n}(dx) + \int_{B_{\eps}} \epsilon_n^{-1} T_{\epsilon_n}^* 1_{B_{\epsilon}}(x) Z_{t-\delta_n}(dx).
\end{align}
Consider the second term in \eqref{e_liminfsemigroup1}. For $\eps ' \in (0,\eps)$ and $x \in B_{\eps'}$, we have $B(x, \eps - \eps') \subset B_\epsilon$, and hence
\begin{equation}
T_{\eps_n}^* 1_{B_\eps}(x) = \bP^{W^*}_x(W^*_{\eps_n} \in B_\eps) \geq \bP^{W^*}_0(W_{\eps_n}^* \in B(0,\eps - \eps')). \nonumber
\end{equation}
Since the right hand side of the above converges to $1$ as $\epsilon_n \to 0$, it follows that $T_{\eps_n}^* 1_{B_\eps} \to 1$ uniformly on $B_{\eps'}$. Consequently, $T_{\eps_n}^* 1_{B_\eps} (x) \geq \frac 1 2$ for all $x \in B_{\eps'}$ for sufficiently large $n$. We conclude that
\begin{align}
\liminf_{n \to \infty} \eps_n^{-1} \int_{B_{\eps}} T_{\eps_n}^* 1_{B_\eps} (x) Z_{t-\delta_n} (dx) \geq \infty \cdot \int_{B_{\eps'}} Z_{t}(dx) \,\, \text{ a.s.}, \nonumber
\end{align}
for every $\eps' \in (0,\eps)$, since $Z_{t-\delta_n} \to Z_t$ weakly and hence $\liminf_{n \to \infty} Z_{t-\delta_n}(B_{\eps'}) \geq Z_t(B_{\eps'})$. By inner regularity (continuity of probability) of $Z_t$, taking $\epsilon' \uparrow \eps$ gives
\begin{align}\label{e_liminfsemigroup2}
\liminf_{n \to \infty} \eps_n^{-1} \int_{B_{\eps}} T_{\eps_n}^* 1_{B_\eps} (x) Z_{t-\delta_n} (dx) \geq \infty \cdot Z_t(B_{\eps}) \,\, \text{ a.s.}
\end{align}

We now consider the first term in the right hand side of \eqref{e_liminfsemigroup1}. Since $B$ is bounded, for $\epsilon' \in (0,\epsilon)$ we can choose a non-negative function $\varphi \in C^\infty_c$ satisfying
\begin{equation}
1_{B_{\eps'}} \leq \varphi \leq 1_{B_\eps} \nonumber
\end{equation}
such that $\varphi$, $D\varphi$ and $D^2 \varphi$ vanish on $B_\eps^c$. ($D \varphi$ and $D^2 \varphi$ denote the gradient and Hessian of $\varphi$, respectively.) Note that we also have
\begin{equation} \label{eq_semigroupinequality}
T_{\eps_n}^* 1_{B_{\eps'}} \leq T_{\eps_n}^*  \varphi \leq T_{\eps_n}^*  1_{B_\eps}.
\end{equation}
Since $\varphi$ is in the domain of $\cA^*$ (viewed as an operator on $C_0$), we have
\begin{equation}
\eps_n^{-1} (T_{\eps_n}^*\varphi - \varphi) \to A^* \varphi  \nonumber
\end{equation}
uniformly as $n \to \infty$. However, because $\varphi(x) = 0$ for $x \in B_\eps^c$, this implies
\begin{equation}
\eps_n^{-1} T_{\eps_n}^*\varphi\to A^* \varphi \quad \text{uniformly on} \,\, B_\eps^c. \nonumber
\end{equation}
It follows from \eqref{eq_semigroupinequality}, the above, and the weak convergence of $Z_{t-\delta_n}$ to $Z_t$ that
\begin{align}
\liminf_{n\to \infty} \int_{B_\eps^c} \eps_n^{-1} T_{\eps_n}^* 1_{B_\eps}(x) Z_{t-\delta_n}(dx) & \geq \int_{B_\eps^c} A^* \varphi(x) Z_t(dx). \nonumber
\end{align}
Because $\varphi$, $D\varphi$ and $D^2 \varphi$ vanish on $B_\eps^c$, we obtain from \eqref{eq_adjointgenerator} that
\begin{equation}
A^*\varphi(x) = \int_{\R^d} \varphi(x-y) \nu(dy) \nonumber
\end{equation}
for all $x \in B_\eps^c$. We thus obtain that
\begin{align}
\liminf_{n\to \infty} \int_{B_\eps^c} \eps_n^{-1} T_{\eps_n}^* 1_{B_\eps}(x) Z_{t-\delta_n}(dx) &\geq \int_{B_\eps^c} \int_{\R^d} \varphi(x-y) \nu(dy) Z_t(dx) \nonumber
\\& \geq \int_{B_\eps^c} \int_{\R^d} 1_{B_{\eps'}}(x-y) \nu(dy) Z_t(dx). \nonumber
\end{align}
Taking $\epsilon' \uparrow \epsilon$ and applying monotone convergence, we obtain
\begin{align}
\liminf_{n\to \infty} \int_{B_\eps^c} \eps_n^{-1} T_{\eps_n} 1_{B_\eps} Z_{t-\delta_n}(dx) &\geq  \int_{B_\eps^c} \int_{\R^d} 1_{B_{\eps}}(x - y) \nu(dy) Z_t(dx). \nonumber
\end{align}
Combined with \eqref{e_liminfsemigroup1} and \eqref{e_liminfsemigroup2}, this yields
\begin{align}
\liminf_{n \to \infty} \epsilon_n^{-1} \langle T_{\epsilon_n} 1_{B_\epsilon} , Z_{t-\delta_n}\rangle \geq \int_{B_\eps^c}\int_{\R^d} 1_{B_\eps}(x-y) \nu(dy) Z_t(dx) + \infty \cdot Z_t(B_\eps). \nonumber
\end{align}
To obtain the desired result, we observe that
\begin{equation}
\int_{B_\eps} \int_{\R^d} 1_{B_\eps}(x-y) \nu(dy) Z_t(dx) > 0 \Rightarrow Z_t(B_\eps) > 0, \nonumber
\end{equation}
and hence
\begin{align}
&\int_{B_\eps^c}\int_{\R^d} 1_{B_\eps}(x-y) \nu(dy) Z_t(dx) + \infty \cdot Z_t(B_\eps) \nonumber
\\ &\hspace{1 cm}= \int_{B_\eps^c}\int_{\R^d} 1_{B_\eps}(x-y) \nu(dy) Z_t(dx) + \infty \cdot Z_t(B_\eps) + \int_{B_\eps} \int_{\R^d} 1_{B_\eps}(x-y) \nu(dy) Z_t(dx) \nonumber
\\ &\hspace{1 cm}= (\nu * Z_t)(B_\eps) + \infty \cdot Z_t(B_\eps), \nonumber
\end{align}
which completes the proof.
\end{proof}

The last lemma we need to prove our main result is the following result that allows us to handle general $\Lambda$-coalescents. In order to state it, we review the notion of dust. An exchangeable integer partition is said to have {\it dust} if a positive proportion of the integers belong to singleton blocks. In fact, all blocks with asymptotic frequency equal to zero are singleton blocks, and an exchangeable partition has singleton blocks if and only if a positive proportion of the integers belong to singleton blocks. (See Proposition 2.8 of Bertoin \cite{Bertoin_coag}.) Thus, for $\Pi(t)$, having no dust implies that $\lim_{b\to 0^+} N_t(b) = N_t$. (Note that this holds if $N_t < \infty$ or $N_t = + \infty$.) Hereafter, we say that the $\Lambda$-coalescent has no dust if $\Pi(t)$ has no dust a.s. for all $t>0$, in which case we have $\lim_{b \to 0^+} N_t(b) = N_t$ a.s. for all $t>0$. Note that any $\Lambda$-coalescent that comes down from infinity does not have dust; a necessary and sufficient condition for a $\Lambda$-coalescent to have no dust is given in Theorem~8 of \cite{P1999}.

Recall the function $v(t)$ defined in \eqref{def_v}.


\begin{lemma} \label{lemma_specialsequence} Let $\Lambda \in \cM_f([0,1])$ be such that the $\Lambda$-coalescent a.s. has no dust. Then for any $a \in (0,1)$, there exist positive sequences $(\delta_n)_{n \geq 1}$ and $(b_n)_{n \geq 1} \subset [0,1/2]$ and a constant $c>0$ such that $\delta_n \to 0$ and $b_n \to 0$ as $n\to \infty$ and the following hold:\begin{itemize}
\item $\delta_{n+1} \leq a \delta_n$ for all $n \in \N$.
\item $\bP^\Pi(N_{\delta_n /2}( 2 b_n) \geq c \delta_n^{-1}) \geq 1/4$ for all $n \in \N$.
\item $n \to b_n v(\delta_{n+1})$ is increasing and $b_n v(\delta_{n+1}) \geq n$ for all $n \in \N$.
\end{itemize}
\end{lemma}

\begin{proof}
We assume that $v(\delta) < \infty$ for all $\delta>0$, as (under the assumption that $\Pi_t$ has no dust) the lemma is trivial otherwise. Theorem~\ref{theorem_blocks} implies that $N_\delta / v(\delta) \to 1$ in distribution as $\delta \to 0$. Hence there exists $\delta_0 > 0$ such that if $0< \delta \leq \delta_0$, then $N_\delta \geq v(\delta) / 2$ with probability at least $1/2$. By \eqref{e_blocknum_bd}, there is a constant $c>0$ such that $v(\delta) \geq c\delta^{-1}$. Hence
\begin{equation}
\bP^\Pi( N_\delta \geq  c \delta^{-1} ) \geq \frac 1 2  \nonumber
\end{equation}
for any $\delta \leq \delta_0$. We also have
\begin{equation}
\lim_{b \to 0^+} N_\delta(b) = N_\delta \nonumber
\end{equation}
almost surely for any $\delta > 0$, since $\Pi(t)$ does not have dust. Consequently,
\begin{equation}
\lim_{b \to 0^+} \bP^\Pi ( N_\delta(b) \geq  c \delta^{-1}/2 ) \geq 1/2. \nonumber
\end{equation}
It follows that, given $\delta \leq \delta_0$ there exists $b^*(\delta)>0$ such that
\begin{equation} \label{e_sequencelemmaaux}
\bP^\Pi (N_\delta( 2b^*(\delta)) \geq  c \delta^{-1}/2) \geq 1/4.
\end{equation}
Now fix $\delta_1 < \delta_0$ and choose $b_1 = b^*(\delta_1/2)$. Since $\lim_{s \to 0^+} v(s) = \infty$, we can choose $\delta_2 \in (0, a \delta_1]$ to be sufficiently small so that $v(\delta_2) b_1 \geq 2$. We then choose $b_2 = b^*(\delta_2/2)$. We continue the construction sequentially, always taking $b_n = b^*(\delta_n / 2)$ and choosing $\delta_{n+1} \in (0, a \delta_n]$ such that $v(\delta_{n+1}) b_{n} \geq n+1$, which guarantees that $\delta_{n+1} \leq a \delta_n$ and $v(\delta_{n+1}) b_n \to \infty$. This also implies that $v(\delta_n) \to \infty$, and hence $\delta_n \to 0$ as $n \to \infty$. The fact that $b_n = b^*(\delta_n/2)$ for all $n$ ensures that $N_{\delta_n/2}( 2 b_n) \geq  c \delta_n^{-1}$ with probability at least $1/4$ for all $n$ by \eqref{e_sequencelemmaaux}, and the proof is complete.
\end{proof}

We now prove Theorem~\ref{thm_main} under the assumption that $\Pi(t)$ does not have dust. (This is the important case, as indeed the proof is very simple when $\Pi(t)$ has dust. We discuss this afterwards.) The idea of the proof is as follows. For an open ball $B$, we use Lemmas~\ref{lemma_masslwrbd} and \ref{lemma_liminfgenerator} to show that $Z_{t-\delta_{n+1}}(B) \geq b_n$ infinitely often a.s. on $\{\nu * Z_t(B) > 0\}$, where $\delta_n$ and $b_n$ are chosen as in Lemma~\ref{lemma_specialsequence}. We then estimate $Z_t(B)$ conditionally given $Z_{t-\delta_{n+1}}$; the corresponding ancestral representation of $Z_t$ has order $v(\delta_{n+1})$ ancestors, so when $Z_{t-\delta_{n+1}}(B) \geq b_n$, the expected number of ancestors originating in $B$ is of order at least $b_n v(\delta_{n+1})$, which goes to infinity as $n \to \infty$. Each cluster originating in $B$ locally generates mass (conditionally) independently with some positive probability. Therefore, the probability that a cluster from time $t-\delta_{n+1}$ originates in $B$ \textit{and} generates mass in $B$ at time $t$ converges can be shown to converge to $1$ using the estimates developed before.

\begin{proof}[Proof of Theorem~\ref{thm_main}] We fix $t>0$ and $\mu \in \cM_1$. Let $B \subset \R^d$ be bounded. For the time being we allow an arbitrary set, but we will consider $B_\eps$ for $\eps > 0$ and the reader may think of $B_\eps$ as an open ball. By Lemma~\ref{lemma_specialsequence}, there exists a constant $c_0>0$ and positive sequences $(\delta_n)_{n \geq 1}$ and $(b_n)_{n \geq 1} \subset [0,1/2]$, both converging to zero, satisfying the following:
\begin{align}
& \delta_{n+1} \leq \frac{\delta_n}{3} \text{ for all } n \in \N, \label{e_sequence1}
\\& \bP^\Pi(N(\delta_n /2, 2 b_n) \geq c_0 \delta_n^{-1}) \geq \frac 1 4, \label{e_sequence2}
\\& b_n v(\delta_{n+1}) \text{ is increasing and } b_n v(\delta_{n+1}) \geq n \text{ for all } n \in \N. \label{e_sequence3}
\end{align}
Without loss of generality we assume that $\delta_1 < t$.

Let $\epsilon >0$. For $n \in \N$ we define the event $I_n$ by
\begin{equation}
I_n := \{Z_{t-\delta_{n+1}}(B_\eps) \geq b_n /2\}. \nonumber
\end{equation}
First, we estimate the probability of $I_{n}$ conditional on $\cF_{t-\delta_n}$. Let $\mu_n := Z_{t-\delta_n}$ under $\bP^Z_\mu$. By the Markov property,
\begin{equation}
\bP^Z_\mu(I_{n} \, | \, \cF_{t-\delta_n}) = \bP^Z_{\mu_n} ( Z_{\delta_n - \delta_{n+1}}(B_\eps) \geq b_n/2). \nonumber
\end{equation}
We let $\eps' \in (0,\eps)$ and define
\begin{equation}
A_n := \{ Z_{\delta_n - \delta_{n+1}}(B_\eps) \geq b_n \cdot p(\delta_n / 2, \eps - \eps') \} \nonumber.
\end{equation}
By the definition of $p(\delta_n/2,\eps-\eps')$ (see \eqref{def_pdeltaeps}) and the continuity in probability of L\'evy processes, $\lim_{s \to 0^+} p(s,r) = 1$ for any $r>0$. Hence there exists $n_0(\eps- \eps') \in \N$ such that
\begin{equation} \label{e_pf_locallevy}
n \geq n_0(\eps-\eps') \Rightarrow p(\delta_n/2,\eps-\eps') \geq \frac 1 2.
\end{equation}
This implies that
\begin{equation} \label{e_pf_InAn}
\bP^Z_\mu (I_{n} \, | \, \cF_{t-\delta_n}) \geq \bP^Z_{\mu_n} (A_n)  \,\,\text{ for all } n \geq n_0(\eps - \eps').
\end{equation}
We now apply Lemma~\ref{lemma_masslwrbd} to estimate $\bP^Z_{\mu_n}(A_n)$. In the notation of that lemma we take $t_0 = \delta_n - \delta_{n+1}$, $\delta = \delta_n / 2$ and $b = b_n$. Observe that in this case, $t_0 - \delta = \delta_n - \delta_{n+1} - \delta_n /2 = \delta_n / 2 - \delta_{n+1} > 0$ by \eqref{e_sequence1}. We define
\begin{equation}
m_n = T_{\delta_n/2 - \delta_{n+1}} \mu_n (B_{\eps'}). \nonumber
\end{equation}
Finally, we observe that $B_\eps = (B_{\eps'})_{\eps - \eps'}$. Applying Lemma~\ref{lemma_masslwrbd} with these parameters, we obtain
\begin{equation} \label{e_pf_smallmeanbd1}
\bP^Z_{\mu_n}( A_n ) \geq \frac{\E^\Pi(1 - e^{-\theta_n})}{8(\sigma + 1)} \left[ 1 \wedge m_n (\delta_n / 2 - \delta_{n+1} )^{-1} \right],
\end{equation}
where $\theta_n = m_n + \sigma (\delta_n / 2 - \delta_{n+1}) p(\delta_n / 2,\eps - \eps') N_{\delta_n/2}(2b_n)$ and we recall that $N_{\delta_n/2}(2b_n)$ denotes the number of blocks in the $\Lambda$-coalescent at time $\delta_n / 2$ with at asymptotic frequency at least $2b_n$. By \eqref{e_sequence1} and \eqref{e_pf_locallevy}, if $n\geq n_0(\eps-\eps')$ then
\begin{align}
\theta_n &= m_n+ \sigma (\delta_n / 2 - \delta_{n+1}) p(\delta_n / 2,\eps - \eps') N_{\delta_n/2}(2b_n) \nonumber
\\ &\geq \frac{\sigma \delta_n}{12} N_{\delta_n/2}(2b_n). \nonumber
\end{align}
If $N_{\delta_n/2}(2b_n) \geq c_0 \delta_n^{-1}$, $\theta_n$ is bounded below by a constant (and hence so is $1 - e^{-\theta_n}$). Thus it follows from \eqref{e_sequence2} that $\E^\Pi(1 - e^{-\theta_n}) \geq c_1 > 0$ for some $c_1 > 0$, for all $n \geq n_0(\eps - \eps')$. Combined with \eqref{e_pf_smallmeanbd1}, this implies that
\begin{equation}
\bP^Z_{\mu_n}( A_n) \geq c_2 \left[ 1 \wedge m_n (\delta_n / 2 - \delta_{n+1} )^{-1} \right] \nonumber
\end{equation}
for a constant $c_2 > 0$, for all $n \geq n_0(\eps- \eps')$. Returning to \eqref{e_pf_InAn}, recalling that $\mu_n = Z_{t-\delta_n}$, and writing $m_n$ in terms of its definition, we have shown that for sufficiently large $n$,
\begin{equation} \label{e_pf_condbound1}
\bP^Z_\mu(I_{n} \, | \, \cF_{t-\delta_n}) \geq c_2 \left[ 1 \wedge \frac{T_{\delta_n/2 - \delta_{n+1}} Z_{t - \delta_n} (B_{\eps'})}{\delta_n / 2 - \delta_{n+1}} \right].
\end{equation}
Observe that the main term in the above is the quantity from Lemma~\ref{lemma_liminfgenerator} with $\eps_n = \delta_n/2 - \delta_{n+1}$. By Lemma~\ref{lemma_liminfgenerator},
\begin{equation}
\liminf_{n \to \infty} (\delta_n / 2 - \delta_{n+1} )^{-1} T_{\delta_n/2 - \delta_{n+1}} Z_{t - \delta_n} (B_{\eps'}) \geq \nu * Z_t(B_{\epsilon'}) + \infty \cdot Z_t (B_{\epsilon'}). \nonumber
\end{equation}
Taking only the first term in the above, \eqref{e_pf_condbound1} implies that
\begin{equation}
\liminf_{ n \to \infty} \bP^Z_\mu(I_{n} \, | \, \cF_{t-\delta_n}) \geq c_2 \left[ 1 \wedge \nu * Z_t(B_{\epsilon'}) \right]. \nonumber
\end{equation}	
As this holds for all $\eps' \in (0,\eps)$ and $\lim_{\eps' \uparrow \eps} Z_t(B_{\eps'}) = Z_t(B_\eps)$ (by inner regularity of the Borel measure $Z_t$), taking $\eps' \uparrow \eps$ implies that
\begin{equation} \label{e_pf_condbound2}
\liminf_{ n \to \infty} \bP^Z_\mu(I_{n} \, | \, \cF_{t-\delta_n}) \geq c_2 \left[ 1 \wedge \nu * Z_t(B_{\epsilon}) \right].
\end{equation}
If $\nu * Z_t(B_{\epsilon}) > 0$, the right hand side above is positive and bounded below. The extended (second) Borel-Cantelli lemma for conditionally independent events then implies that $I_{n}$ occurs infinitely often. That is,
\begin{equation}
\text{$I_n$ occurs infinitely often almost surely on $\{\nu * Z_t(B_\eps) > 0\}$.} \nonumber
\end{equation}

We now specialize our argument to open balls and complete the proof. Let $x_0 \in \R^d$ and $r>0$. If we take $B = \{x_0\}$ and $\eps = r' \in (0,r)$,
\begin{equation} \label{e_pf_ballmassio}
\text{$I_n(x_0,r')$ occurs infinitely often almost surely on $\{\nu * Z_t(B(x_0,r')) > 0\}$,}
\end{equation}
where $I_n(x_0,r') = \{Z_{t-\delta_{n+1}}(B(x_0,r')) \geq b_n /2\}$. We define $n_1 = \inf \{n : I_n(x_0,r') \text{ occurs}\}$ and $n_k = \inf \{n > n_{k-1}: I_n(x_0,r') \text{ occurs}\}$ for all $k \geq 2$, with the convention that $\inf \emptyset = \infty$. Then on $\{ \nu * Z_t(B(x_0,r)) > 0\}$, $n_k < \infty$ for all $k$ almost surely. For $k\in \N$ we define the random times
\begin{equation}
\tau_k = \begin{cases} t - \delta_{n_k + 1} &\text{ if } n_k < \infty, \\
t &\text{ if } n_k = \infty.  \end{cases} \nonumber
\end{equation} 
Indeed, it is straightforward to show that $\tau_k$ is a $(\cF_s)$-stopping time. By the previous discussion, \eqref{e_pf_ballmassio} is equivalent to
\begin{equation} \label{e_stoppingtimes_as}
n_k < \infty \text{  and  } \tau_k = t - \delta_{n_k + 1} \,\, \text{ for all $k \in \N$ a.s. on $\{ \nu * Z_t(B(x_0,r')) > 0\}$.}
\end{equation}
We apply the strong Markov property at $\tau_k$ and use the lookdown representation to obtain that, when $n_k<\infty$,
\begin{equation} \label{e_pf_strongMarkovlookdown}
\bP^Z_\mu( Z_t(B(x_0,r')) > 0 \, | \, \cF_{\tau_k}) = \cL_{\tilde{\nu}_{k}}(Z_{\delta_{n_k + 1}}(B(x_0,r'))>0),
\end{equation}
where we define $\tilde{\mu}_k := Z_{\tau_k} = Z_{t - \delta_{n_k+1}}$. We estimate the right-hand side by conditioning on the ancestral coalescent. Omitting the time dependence of the blocks, we write $\Pi^{\delta_{n_k+1}}(\delta_{n_k+1}) = \{ \pi_i : i \in [N_{\delta_{n_k+1}}]\}$, and recall from Lemma~\ref{lemma_labels} that the ancestor of $\pi_i$ is the level $i$ individual. For $i \in [N_{\delta_{n_k+1}}]$, by Lemma~\ref{lemma_local_cluster_mass}, and again using the fact that all blocks have positive asymptotic frequency when $\Pi(t)$ comes down from infinity, we have
\begin{equation} \label{e_pf_goodclusterbd}
\cL_{\tilde{\mu}_k} (Z_{i,\delta_{n_k+1}}(\delta_{n_k+1}, B(x_0,r)) > 0 \, | \, \cF^\Pi_{\delta_{n_k+1}}, X(0)) \geq  \bP^W_{X_i(0)}(W_{\delta_{n_k+1}} \in B(x_0,r))/2. \nonumber
\end{equation}
If $y \in B(x_0,r-r')$, we have
\[ \bP^W_{y}(W_{\delta_{n_k+1}} \in B(x_0,r)) \geq p(\delta_{n_k+1}, r-r')\]
by spatial homogeneousness of the L\'evy process. The initial particle positions are i.i.d. with distribution $\tilde{\mu}_k$, and hence $X_i(0) \in B(x_0, r')$ with probability $\tilde{\mu}_k(B(x_0,r'))$. Hence, for each $i \in [N_{\delta_{n_k+1}}]$, 
\begin{equation}
\cL_{\tilde{\mu}_k} (Z_{i, \delta_{n_k}+1}(\delta_{n_k+1}, B(x_0,r)) > 0 \, | \, \cF^\Pi_{\delta_{n_k+1}}) \geq \frac{p(\delta_{n_k+1} ,r-r') \tilde{\mu}_k(B(x_0,r'))}{2}. \nonumber
\end{equation}
By definition, $Z(\delta_{n_k+1} , B(x_0,r')) \geq Z_{i,\delta_{n_k+1}}(\delta_{n_k+1} , B(x_0,r'))$ for each $i \in [N_{\delta_{n_k+1}}]$, and hence for the former to be positive we simply require that $Z_{i,\delta_{n_k+1}}(\delta_{n_k+1} , B(x_0,r'))$ is positive for some $i \in [N_{\delta_{n_k+1}}]$. By Lemma~\ref{lemma_indep_ancestors}, the ancestral clusters $\{Z_{i,\delta_{n_k+1}}(\delta_{n_k+1}, \cdot) : i \in [N_{\delta_{n_k+1}}]\}$ are conditionally independent given $\cF^\Pi_{\delta_{n_k+1}}$, and hence
\begin{equation}
\cL_{\tilde{\mu}_k} (Z(\delta_{n_k+1}, B(x_0,r)) > 0 \, | \, \cF^\Pi_{\delta_{n_k+1}}) \geq 1 - \left(1 - \frac{p(\delta_{n_k+1} ,r-r') \tilde{\mu}_k (B(x_0,r'))}{2} \right)^{N_{\delta_{n_k+1}}}. \nonumber
\end{equation}
By definition of $n_k$ and $\tilde{\mu}_k$, $\tilde{\mu}_k(B(x_0,r')) \geq b_{n_k} /2$ on $\{n_k < \infty\}$. Thus, for sufficiently large $k$ and $n_k < \infty$,
\begin{align}
\cL_{\tilde{\mu}_k} (Z(\delta_{n_k+1}, B(x_0,r)) > 0 \, | \, \cF^\Pi_{\delta_{n_k+1}}) &\geq 1 - \left(1 - \frac{p(\delta_{n_k+1}, r - r') b_{n_k}}{4} \right)^{N_{\delta_{n_k+1}}}   \nonumber
\\ &\geq  1 - e^{- b_{n_k} N_{\delta_{n_k+1}}/8}  \nonumber
\\ &\geq (1 - e^{- b_{n_k} v(\delta_{n_k+1})/16}) 1(N_{\delta_{n_k+1}} \geq v(\delta_{n_k+1}) / 2) \nonumber
\end{align}
The second inequality implicitly uses the fact that $p(\delta_{n+1}, r - r') \geq 1/2$ for sufficiently large $n$ (c.f. \eqref{e_pf_condbound2}). Taking the expectation of the above, we obtain
\begin{align} \label{e_pf_clustersbd}
\cL_{\tilde{\mu}_k} (Z(\delta_{n_k+1}, B(x_0,r)) > 0)  = (1 - e^{-b_{n_k} v(\delta_{n_k+1}) / 16}) \bP^\Pi(N_{\delta_{n_k+1}} \geq v(\delta_{n_k+1})/2).
\end{align}
By Theorem~\ref{theorem_blocks},
\begin{equation}
\frac{N_s}{v(s)}  \to 1 \nonumber
\end{equation}
in distribution as $s \downarrow 0$. In particular,  there exists a non-decreasing function $s \to \tilde{\eps}(s)$ such that $\lim_{s \to 0^+} \tilde{\eps}(s) = 0$ and
\begin{equation}
\sup_{u \in (0,s]}\bP^\Pi(N_u \geq v(u)/2) \geq 1 - \tilde{\eps}(s).  \nonumber
\end{equation}
By \eqref{e_pf_clustersbd}, for $n_k < \infty$ we have
\begin{align} \cL_{Z_{\tau_k}} (Z(\delta_{n_k+1}, B(x_0,r)) > 0) \geq (1 - e^{-b_{n_k} v(\delta_{n_k+1}) / 16}) (1 - \tilde{\eps}(\delta_{n_k+1})). \nonumber
\end{align}
By \eqref{e_sequence3}, we have $b_{n_k} v(\delta_{n_k +1}) \geq n_k \geq k$ on $\{n_k < \infty\}$. Combining this with the above and \eqref{e_pf_strongMarkovlookdown}, we obtain that on the event $\{n_k < \infty\}$,
\begin{align}
\bP(Z_t(B(x_0,r) > 0 \, | \, \cF_{\tau_k}) \geq (1 - e^{-k / 16}) (1 - \tilde{\eps}(\delta_{k+1})). \nonumber
\end{align}
We now compute
\begin{align}
\bP^Z_\mu( Z_t(B(x_0,r)> 0, n_k < \infty) &= \E^Z_\mu( P(Z_t(B(x_0,r) >0, n_k < \infty \, | \, \cF_{\tau_k}))  \nonumber
\\& =\E^Z_\mu( P(Z_t(B(x_0,r) >0 \, | \, \cF_{\tau_k}) 1_{\{n_k < \infty\}})  \nonumber
\\ &\geq (1 - e^{-k / 16}) (1 - \tilde{\eps}(\delta_{k+1})) \bP^Z_\mu(n_k < \infty).
\end{align}
The second line uses the fact that $\{n_k < \infty\} \in \cF_{\tau_k}$, and the third line uses the bound for the conditional probability derived above. Let $\mathcal{E} = \cap_{k=1}^\infty \{n_k < \infty\}$ and note that, since $\{n_{k+1} < \infty \} \subseteq \{n_{k} < \infty\}$, $1_\mathcal{E} = \lim_{k \to \infty} 1_{\{n_k < \infty\}}$ a.s. Taking $k \to \infty$ in the above, it follows that
\begin{equation}
\bP^Z_\mu( Z_t(B(x_0,r)> 0, \mathcal{E}) \geq \bP^Z_\mu(\mathcal{E}), \nonumber
\end{equation}
and hence the two sides are equal. By \eqref{e_stoppingtimes_as}, $\{\nu * Z_t(B(x_0,r'))>0\} \subseteq \mathcal{E}$, and therefore
\begin{equation}
\bP^Z_\mu(Z_t(B(x_0,r)) = 0, \nu * Z_t(B(x_0,r')) > 0 ) = 0. \nonumber
\end{equation}
Regularity of $\nu * Z_t$ implies that $\nu * Z_t(B(x_0,r)) > 0 $ if and only if $\nu * Z_t(B(x_0,r')) > 0 $ for some $r' < r$. We thus deduce that
\begin{equation}
\bP^Z_\mu(Z_t(B(x_0,r)) = 0, \nu * Z_t(B(x_0,r)) > 0 ) = 0. \nonumber
\end{equation}
To complete the argument, we observe that the implication holds simultaneously for the countable collection of open balls with rational centres and radii almost surely. In particular, the event that there exists $x_0 \in \Q^d$ and $r \in \Q^+$ such that $\nu*Z_t(B(x_0,r))>0$ but $Z_t(B(x_0,r)) = 0$ has probability zero. On the complement of this event, an elementary topological argument yields that $S(\nu * Z_t) \subseteq S(Z_t)$. Hence
$S(\nu * Z_t) \subseteq S(Z_t)$ almost surely. Iterating the previous procedure, it follows that for $k \in \N$, the same is true of the $k$-fold convolution of $\nu$ with $Z_t$, and the proof is complete.
\end{proof}

We conclude by remarking on the case that $\Pi(t)$ has dust, i.e. when a positive proportion of the integers belong to singleton blocks in $\Pi(t)$. This case is much easier and we just sketch the proof. First, we claim that without loss of generality we may assume the mutation process $W_t$ is a pure jump L\'{e}vy process. If $W_t$ has a Brownian component, then the argument of Birkner and Blath \cite{BB09b} when $\Pi(t)$ does not come down from infinity applies, and it follows that $S(Z_t) = \R^d$ almost surely. If $W_t$ has a drift, then there is a driftless L\'evy process $\hat{W}_t$ and $a \in \R^d$ such that $W_t = at + \hat{W}_t$. One can construct coupled Fleming-Viot processes associated to mutation processes $W_t$ and $\hat{W}_t$, denoted $Z_t$ and $\hat{Z}_t$ with the obvious correspondence, such that $
Z_t = \hat{Z}_t + at$
for all $t>0$, where $\hat{Z}_t + at$ is the measure $\hat{Z}_t$ shifted by $at \in \R^d$. Thus $S(Z_t) = S(\hat{Z}_t + at)$, and it follows that if $S(\nu * \hat{Z}_t ) \subseteq S(\hat{Z}_t)$, then $S(\nu * Z_t) \subseteq S(Z_t)$, so it suffices to consider $\hat{Z}_t$. 

Let $Z_0 = \mu \in \cM_1$. Suppose that $\Pi(t)$ has dust and assume in addition that $\Lambda(\{1\}) = 0$. Under this additional assumption, it is not hard to see that for all $t>0$ there exists $\alpha(t) \in (0,1]$ such that a subset of the integers with frequency $\alpha(t)$ belong to singleton blocks. The individuals in singleton blocks evolve conditionally as independent copies of $W_t$ and exchangeability implies that their initial positions are uniformly distributed from $\mu$. The law of large numbers then implies that $Z_t$ has a component which is the heat flow associated to $W_t$, started from initial measure $\alpha(t) \mu$. In particular, it follows that $S(T_t \mu) \subseteq S(Z_t)$. On the other hand, it is clear from the lookdown construction that $S(Z_t) \subseteq S(T_t \mu)$. Thus we must have $S(Z_t) = S(T_t \mu)$. The claim then follows from the observation that, for any measure $\mu$, $S(T_t \mu) = S(\nu * T_t \mu)$, which is a consequence of $W_t$ being a pure jump process with L\'evy measure $\nu$.

Now suppose that $\Lambda(\{1\}) = c >0 $. Then there is a rate $c$ Poisson process with jump times $0<t_1 < t_2 < \cdots$ such that at time $t_k$, $Z_t$ jumps to $\delta_{x_k}$, where $x_k$ is a sample from $Z_{t_k-}$. Thus with probability one there is some $k^* \in \{0\} \cup \N$ such that $t_{k^*} < t < t_{k^* + 1}$ (we define $t_0 = 0$). The same reasoning as above then implies $S(Z_t) = S(T_{t - t_{k^*}} \delta_{x_k})$, and the result follows from the same argument.


\begin{thebibliography}{99}


\bibitem{Ald85}	
\textsc{Aldous, D.J.}, Changeability and related topics. 	\textit{\'Ecole d'\'et\'e de probabilit\'es de Saint-Flour, XIII }, Lecture Notes in Math. \textbf{1117}, 1985.
	
\bibitem{BBL10}
\textsc{Berestycki, J., Berestycki, N., and Limic, V.}, The $\Lambda$-coalescent speed of coming down from infinity.
\textit{Ann. Probab.}  	\textbf{38}, (2010)  207-233.

\bibitem{BBL14}
 \textsc{Berestycki, J., Berestycki, N., and Limic, V.}, A small-time coupling
between $\Lambda$-coalescents and branching processes. \textit{Ann. Probab.}  	\textbf{24}, (2014)  449-475.

\bibitem{Berestycki_notes}	
\textsc{Berestycki, N.}, Recent progress	in coalescent theory. arXiv:0909.3985v1.
	
\bibitem{Bertoin96}	\textsc{Bertoin, J.}, L\'evy processes.  Cambridge University Press, Cambridge, 1996.

\bibitem{Bertoin_coag} \textsc{Bertoin, J.}, Random Fragmentation and Coagulation Processes. Cambridge University Press, Cambridge, 2006.


\bibitem{BLG2003} \textsc{Bertoin, J. and Le Gall, J.F.}, Stochastic flows associated to coalescent processes, \textit{Prob. Theory Relat. Fields} \textbf{126}, (2003) pp. 261-288.


\bibitem{BLG2005} \textsc{Bertoin, J. and Le Gall, J.F.},
Stochastic flows associated to coalescent processes III: Stochastic differential equations, \textit{Ann. Inst. Henri Poincar\'e Probabilit\'es et Statistiques} \textbf{41}, (2005) 307-333.


\bibitem{BG2006} \textsc{Bertoin, J. and Le Gall, J.F.}, Stochastic flows associated to coalescent processes III: Limit theorems, \textit{Illinois J. Math.} \textbf{50}, (2006) pp. 147-181.


	
\bibitem{BBCEMSW05}	
\textsc{Birkner, M., Blath, J., Capaldo, M., Etheridge, A.,  M\"{o}hle, M., Schweinsberg, J., and Wakolbinger, A.},
 $\alpha$-stable branching and $\beta$-coalescents, \textit{Electron. J. Probab.} \textbf{10}, Paper no. 9, 303–325, (2005).	
	
\bibitem{BB09a}
\textsc{Birkner, M. and Blath, J.}, Measure-valued diffusions, general coalescents and population genetic inference, in: Trends in Stochastic Analysis, LMS 353, Cambridge University Press, 329–363, 2009.

\bibitem{BB09b}
\textsc{Birkner, M. and Blath, J.}, Generalised stable Fleming-Viot processes as flickering random measures. \textit{Electron. J. Probab. } \textbf{84} 2418-2437, 2009.


\bibitem{BBMST}
\textsc{Birkner, M.  Blath, J., M$\ddot{\text{o}}$hle, M.,   Steinr$\ddot{\text{u}}$cken, M. and  Tams, J.}, A modified lookdown
construction for the Xi-Fleming-Viot with mutation and populations with recurrent
bottlenecks process. \textit{ALEA}
\textbf{6}, (2009) 25–61.


\bibitem{Blath09}
\textsc{Blath, J.}, Measure-valued processes, self-similarity and flickering random measures. In Fractal Geometry and Stochastics IV. Progr. Probab.
\textbf{61} 175-196. Birkhauser, Basel, 2009


\bibitem{DH82}
\textsc{Dawson, D. and Hochberg, K.}, Wandering random measures in the Fleming-Viot Model.
\textit{Ann. Probab.}  \textbf{10}, (1982), 554-580.






\bibitem{DK96}
\textsc{Donnelly, P. and Kurtz, T.}, A countable representation of the Fleming-Viot
measure-valued diffusion. \textit{Ann. Probab.} \textbf{24}, (1996) 698-742 .

\bibitem{DK99}
\textsc{Donnelly, P. and Kurtz, T.}, Particle representations for measure-valued population
models. \textit{Ann. Probab.}  \textbf{27}, (1999) 166-205.

\bibitem{Eth00}
\textsc{Etheridge, A.}, An Introduction to Superprocesses.  Amer. Math. Soc., Providence, RI, 2000.


\bibitem{EM1991}
\textsc{Etheridge, A. and March, P.}, A note on superprocesses, \textit{Probab. Theory Rel. Fields} \textbf{ 89},  (1991)
141-148.


\bibitem{EK1993} \textsc{Ethier, S.N. and Kurtz, T.}, Fleming-Viot processes in population genetics, \textit{SIAM J. Control Optim.} \textbf{31}, (1993) 345-386.

\bibitem{EP1991} \textsc{Evans, S. and Perkins, E.}, An absolute continuity result for
measure-valued diffusions and applications, \textit{Trans. Amer. Math. Soc.} \textbf{325}, (1991) pp. 661-682.


\bibitem{FV1979} \textsc{Fleming, W.H. and Viot, M.}, Some measure-valued Markov processes in population genetics theory, \textit{Indian University Mathematics Journal} \textbf{ 28}, (1979) 817-843.






\bibitem{FS04}
\textsc{Fleischmann, K. and  Sturm, A.,} A super-stable motion with infinite mean branching,
\textit{Ann. Inst. H. Poincar\'e Probab. Statist.}  \textbf{40}, no. 5, (2004) 513-537.

\bibitem{H2021}
\textsc{Hughes, T.}, The density of the $(\alpha,d, \beta)$-superprocesses and singular solutions to a fractional non-linear PDE. To appear. \textit{Ann. Inst. H. Poincar\'e Probab. Statist. }  


\bibitem{Kingman}	
\textsc{Kingman, J.F.C.}, The coalescent. \textit{Stochastic Process. Appl.}	\textbf{13}, (1982) 235-248.



\bibitem{Ky14} \textsc{Kyprianou, A.E.} (2014)
Fluctuations of L\'evy Processes with Applications. Springer Berlin Heidelberg.


\bibitem{Li19} \textsc{Li, Z.} Continuous-state branching processes with immigration. arXiv:1901.03521, 2019.

\bibitem{LZ2008} \textsc{Li, Z. and Zhou, X.}, Distribution and propagation properties of superprocesses with general branching mechanisms. \textit{Comm. Stoch. Anal.} \textbf{2}, (2008) No. 3, pp. 469-477.

\bibitem{LZ2012} \textsc{Liu, H. and Zhou, X.}, The compact support property for the $\Lambda$-Fleming-Viot process with underlying Brownian motion, \textit{Electron. J. Probab.} \textbf{17}, (2012) no. 73, 1-20.



\bibitem{LZ2015} \textsc{Liu, H. and Zhou, X.}, Some support properties for a class of $\Lambda$-Fleming-Viot processes, \textit{Ann. Inst. H. Poincar\'e Probab. Statist.}  \textbf{51}, (2015) 1076-1101.



\bibitem{MM2020} \textsc{Mamin, R. and Mytnik, L.}, Absolute continuity of the super-Brownian motion with infinite mean. arXiv:2012.09040v1.

\bibitem{P1988} \textsc{Perkins, E.}, A space-time property of a class of measure-valued branching diffusions, \textit{Trans. Amer. Math. Soc.} \textbf{305}, (1989) 743-795.

\bibitem{P1989} \textsc{Perkins, E.}, The Hausdorff measure of the closed support of super-Brownian motion, \textit{Ann. Henri Poincaré B} \textbf{25}, (1989) 205-224.

\bibitem{P1990} \textsc{Perkins, E.}, Polar sets and multiple points for super-Brownian motion, \textit{Ann. Probab.} \textbf{18}, (1990) 453-491.

\bibitem{P1991}
\textsc{Perkins, E.}, Conditional Dawson-Watanabe processes and Fleming-Viot processes, Seminar in Stochastic Processes, Birkh\"{a}user, pp 142-155, 1991.


\bibitem{P2002} \textsc{Perkins, E.}, \textit{Dawson-Watanabe Superprocesses and Measure-valued
Diffusions}, In: {\it Lectures on Probability Theory and Statistics, Ecole d'Et\'e de probabilit\'es de Saint-Flour XXIX-1999}, Ed. P. Bernard, Lecture Notes in Mathematics {\bf 1781}, 132-335, Springer, Berlin, 2001.


\bibitem{P1999} \textsc{Pitman, J.}, Coalescents with multiple collisions, \textit{Ann. Probab.} \textbf{27}, (1999) 1870-1902.



\bibitem{S1999} \textsc{Sagitov, S.}, The general coalescent with asynchronous mergers of ancestral lines, \textit{J. Appl. Prob.} \textbf{36}, (1999) 1116-1125.


\bibitem{Schw2000a}
\textsc{Schweinsberg, J.}, Coalescents with simultaneous multiple collisions. \textit{Electron. J. Probab.} \textbf{5} (12), (2000) 1-50.

\bibitem{Schw2000}
\textsc{Schweinsberg, J.}, A necessary and sufficient condition for the Lambda-coalescent to come down from infinity, \textit{Electron. Comm. Probab.}
\textbf{5} (1), (2000) 1-11.


\bibitem{Schilling}
\textsc{Schilling, R.}, An Introduction to L\'evy and Feller
Processes. arXiv:1603.00251v2.

\end{thebibliography}
\end{document}